\documentclass[draft]{article}
\usepackage{amsmath,amsfonts,latexsym, amsrefs,amssymb}

\usepackage{color}

\newcommand{\e}{\varepsilon}

\renewcommand{\div}{\mathop{\mathrm{div}}}
\newcommand{\osc}{\mathop{\mathrm{osc}}}

\newcommand{\bke}[1]{\left( #1 \right)}
\newcommand{\bkt}[1]{\left[ #1 \right]}

\newcommand{\norm}[1]{\| #1 \|}

\newcommand{\al}{\alpha}
\newcommand{\de}{\delta}

\newcommand{\ga}{{\gamma}}

\newcommand{\la}{\lambda}
\newcommand{\Om}{{\Omega}}
\newcommand{\om}{{\omega}}

\renewcommand{\th}{\theta}
\newcommand{\td}{\tilde}

\newcommand{\R}{{\mathbb R }}
\newcommand{\N}{{\mathbb N}}
\newcommand{\pd}{{\partial}}
\newcommand{\nb}{{\nabla}}
\newcommand{\lec}{\lesssim}

\newcommand{\ph}{{\varphi}}

\renewcommand{\div}{\mathop{\mathrm{div}}}
\newcommand{\curl}{\mathop{\mathrm{curl}}}

\newcommand{\lv}{{\bar v}}
\newcommand{\lp}{{\bar p}}

\newtheorem{theorem}{Theorem}

\newtheorem{lemma}[theorem]{Lemma}

\newcommand{\qed}{\hfill\fbox{}\par\vspace{0.3mm}}
\newenvironment{proof}{{\bf Proof.}} {\hfill\qed}

\newcommand{\ee}{\end{equation}}

\newcommand{\bey}{\begin{eqnarray}}
\newcommand{\eey}{\end{eqnarray}}
\newcommand{\beyn}{\begin{eqnarray*}}
\newcommand{\eeyn}{\end{eqnarray*}}

\newcommand{\be}{\begin{equation}}

\numberwithin{equation}{section}
\numberwithin{theorem}{section}
\numberwithin{definition}{section}
\begin{document}

%\tableofcontents

%\date{June 18, 2008}
\date{}

\title{Lower bounds on the blow-up rate of the axisymmetric \\ Navier-Stokes equations II}

\author{Chiun-Chuan Chen\thanks{
Department of Mathematics and Taida Institute for Mathematical
Sciences, National Taiwan University and National Center for
Theoretical Sciences, Taipei Office, email:
chchchen@math.ntu.edu.tw}\; , \quad Robert M.
Strain\thanks{Harvard University, email:
strain@math.harvard.edu}\;,\quad Tai-Peng Tsai\thanks{University
of British Columbia, email: ttsai@math.ubc.ca}\;, \quad Horng-Tzer
Yau\thanks{Harvard University, email: htyau@math.harvard.edu} }

\maketitle

\begin{abstract}
Consider axisymmetric strong solutions of the incompressible
Navier-Stokes equations in $\R^3$ with non-trivial swirl.  Let $z$
denote the axis of symmetry and $r$ measure the distance to the
$z$-axis.  Suppose the solution satisfies,
%either $|v (x,t)| \le C_*{|t|^{-1/2}} $ or,
for some $0 \le \e \le 1$, $|v (x,t)| \le C_* r^{-1+\varepsilon }
|t|^{-\varepsilon /2}$ for $-T_0\le t < 0$ and $0<C_*<\infty$ allowed
to be large. We prove that $v$ is regular at time zero.
\end{abstract}

\section{Introduction}

The incompressible Navier-Stokes equations in {\it cartesian coordinates} are given by
\begin{equation} \tag{N-S}
\label{nse}
\pd_t v + (v \cdot \nb)v + \nb p = \Delta v, \quad \div v =0.
\end{equation}
The velocity field is $v(x,t)=(v_1, v_2, v_3):\mathbb{R}^3\times [-T_0,0) \to \R^3$  and
$p(x,t):\mathbb{R}^3\times [-T_0,0) \to \R$ is the pressure.  It is a long standing open question
to determine if solutions with large smooth initial data of finite energy
remain regular for all time.

In this paper we consider the special class of solutions which are {\it axisymmetric}.  This means, in {\it  cylindrical coordinates} $r,\th,z$ with
$(x_1,x_2,x_3)=(r\cos \th,r\sin \th,z)$, that the solution is of the form
\begin{equation} \label{v-cyl}
v(x,t) = v_r(r,z,t)e_r + v_\th(r,z,t)e_\th + v_z(r,z,t)e_z.
\end{equation}
In this coordinate system $r=\sqrt{x_1^2+x_2^2}$.
The components $v_r,v_\th,v_z$ do not depend upon $\th$ and the basis
 vectors $e_r,e_\th,e_z$ are
\[
e_r = \left(\frac {x_1}r,\frac {x_2}r,0\right),\quad
e_\th = \left(-\frac {x_2}r,\frac {x_1}r,0\right),\quad
e_z = (0,0,1).
\]
The main result of our paper
 shows that axisymmetric solutions must blow up faster
than the scale invariant rates which appears in Theorem \ref{mainthm} below.

For $R>0$ define $B(x_0,R)\subset \R^{3}$ as the ball of radius $R$
centered at $x_0$.  The parabolic cylinder is  $Q(X_0,R) = B(x_0,R)\times (t_0-R^2,t_0)\subset
\R^{3+1}$  centered at $X_0=(x_0,t_0)$. If the center is
the origin we use the abbreviations  $B_R=B(0,R)$ and $Q_R=Q(0,R)$.

\begin{theorem}\label{mainthm}
Let $(v,p)$ be an axisymmetric strong solution of the Navier-Stokes equations
\eqref{nse}  in $D=\R^3 \times (-T_0,0)$
%for which $v(x,t)$ is smooth in $x$ and H\"older continuous in $t$.
with initial datum $v|_{t=-T_0}=v^0 \in H^{1/2}$ and $r v^0_\th(r,z)
\in L^\infty$.
Suppose the pressure satisfies $p \in
L^{5/3}(D)$ and $v$ is pointwise bounded by one of the following inequalities:
\begin{equation}\label{thm-a}
|v (x,t)| \le C_*{ |t|^{-1/2}}, \quad (x,t) \in D.
\end{equation}
\begin{equation}\label{thm-b}
\text{There is an $\e\in [0,1]$ such that} \hskip 2cm  |v (x,t)| \le C_*
r^{-1+\varepsilon } |t|^{-\varepsilon /2}, \quad (x,t) \in D.
\end{equation}
The constant $C_*<\infty$ is allowed to be large. Then $v
\in L^\infty(B_R \times [-T_0,0])$ for any $R>0$.
\end{theorem}

We remark that the case $\varepsilon = 0$ is addressed in the
appendix; our proof in that specific case was obtained after a
preprint of \cite{KNSS} had appeared.  The assumption \eqref{thm-a} is
a special case of \eqref{thm-b} with $\e=1$; it is singled out for its
importance. We also remark that the exponent $5/3$ for the norm of $p$
can be replaced, but it is the natural exponent occurring in the
existence theory for weak solutions, see e.g.  \cite{MR673830}.

Recall the natural scaling of Navier-Stokes equations:  If $(v,p)$ is a solution to \eqref{nse}, then for any $\la>0$
the following rescaled pair is also a solution:
\begin{equation}
\label{rescale}
v^\la(x, t) = \la v(\la x, \la^2 t),\quad
p^\la(x,t) = \la^2 p (\la x, \la^2 t).
\end{equation}
Suppose a
solution $v(x,t)$ of the Navier-Stokes equations  blows up at  $X_0=(x_0,t_0)$.  Leray
\cite{JFM60.0726.05} proved that the blow up rate in time is at least
\[
\norm{v(\cdot,t)}_{L^\infty_x} \ge \epsilon {(t_0-t)^{-1/2}}.
\]
Theorem \ref{mainthm} in particular rules out singular axisymmetric
solutions satisfying the similar bound with $\epsilon$ large.

The main idea of our proof is as follows. We shall first prove that
either  \eqref{thm-a} or \eqref{thm-b} with $\e>0$ implies the following
estimate:
\begin{equation}\label{assumption1}
|v| \le  C_* (r^2-t)^{-1/2+2\e}|t|^{-\e } r^{-2\e}.
\end{equation}
This is the content of Section 2 and 3. Note that $\e$ in
\eqref{assumption1} differs from that in \eqref{thm-b}.

If \eqref{assumption1} is satisfied for $\e=0$, the regularity of
$v$ was proved in \cite{CSTY}.  In Sections 4, 5 and 6, we extend
the proof of \cite{CSTY} to include the case \eqref{assumption1} for
$\e > 0$.  Instead of following De Giorgi and Moser's methods
\cite{MR0093649, MR0170091} used in \cite{CSTY}, we now use Nash's
idea \cite{MR855753, MR0100158} to prove the H\"older regularity
(Section 5). This simplifies some iteration arguments in
\cite{CSTY}, but we still use De Georgi-Moser's method in the local
maximum estimate in Section 4. The estimates we obtained in Sections
4 and 5 requires assumptions weaker than \eqref{assumption1}.  Very
recently Koch-Nadirashvili-Seregin-Sverak \cite{KNSS} have sent us a
manuscript that they have proved results similar to Theorem
\ref{mainthm} using a different approach based on Liouville
theorems.

%%%%%%%%%%%%%%%%%%%%%%%%%%%%%%%%%%%%%%%%%%%%%%%%%%%%%%%%%%%%%%%%%%%%%%%%%%%%
\section{The case $|v| \le C |t|^{-1/2}$}
%%%%%%%%%%%%%%%%%%%%%%%%%%%%%%%%%%%%%%%%%%%%%%%%%%%%%%%%%%%%%%%%%%%%%%%%%%%%

Suppose we have $ |v|\le C|t|^{-1/2}$. Our goal is to replace the
singularity of $t$ by singularity in $r$.
We will derive this estimate from the equation for the $\theta$ component of the
vorticity \eqref{q-eq}, which involves a source term $\partial_{z} v_{\theta}^{2}/r$.
Under the assumption $ |v|\le C|t|^{-1/2}$, we have $v_{\theta}^{2} \sim |t|^{-1}$,
singular in $t$ as $t \to 0$. This $t$ singularity can be weaken to $|t|^{-\e}$
after the time integration. Since the equation is scaling invariant, this improvement
in the time singularity has to be offset by the space singularity. This will be achieved
in some weak form in \eqref{eq5}.  Finally, we can transfer estimates on the vorticity to the velocity field and we thus  obtain  the estimate \eqref{assumption1}.

Recall that we always have
the bound $|r v_{\theta}| \le C $ (see Proposition 1 in
\cite{MR1902055}).
Hence for some $C_1>0$
\begin{equation}\label{vbound}
|v_{\theta}|\le C_1 \min( r^{-1}, |t|^{-1/2}), \qquad |v_{r}|+ |v_{z}|
 \le C_1 |t|^{-1/2}.
\end{equation}

For $p,q> 0$,
we will be using the notation
$$
\| v \|_{L^{q, p}_{t, x}(Q_R)}
= \| v \|_{L^{q}_{t}L^{p}_{x}(Q_R)}
= \| v \|_{L^{q}_{t}L^{p}_{x}}
=
\| v \|_{L^{q, p}_{t, x}}.
$$
These are the usual $L^{q,p}$ spaces integrated over space and time.
The domain will be suppressed in our notation below
when there is no risk of ambiguity.

We will next consider the vorticity field $\om =
\curl v$:
\begin{equation}
\om(x,t) = \om_r e_r + \om_\th e_\th + \om_z e_z,
\label{omegaTHETA}
\end{equation}
where
\begin{equation}\label{omega-formula}
\om_r = - \pd_z v_\th, \quad \om_\th = \pd_z v_r - \pd_r v_z,
\quad \om_z = (\pd_r + r^{-1})v_\th.
\end{equation}
We can deduce the following bounds for the $\theta$ component of vorticity.

\begin{lemma} \label{th2-1}
Suppose we have the pointwise bound
\begin{equation}\label{vbound2}
|v(y,s)| \le C_1 |s|^{-1/2},
\end{equation}
in $Q_R(x,t)$. %for $0<|t|<C_2 R^2$.
Then for any $\de \in (0,1)$
we can estimate $\omega$ by
\begin{equation}\label{obound}
\| \omega_\theta \|_{L^{3, 4}_{t, x}} \le C R^{3/4} |t|^{-2/3}+CR^{5/12}|t|^{-1/2}
, \qquad
\| \omega_\theta \|_{L^{6, 8}_{t, x}} \le C R^{3/8}
|t|^{-5/6} + CR^{-7/24}|t|^{-1/2}.
\end{equation}
where the integration is over $Q_{\de R}(x,t)$ and the constant
$C$ depends on $C_1$ and $\de$.
\end{lemma}

\begin{proof}
We can rescale Lemma A.2 of \cite{CSTY} to get, for $\alpha ,q \in (1,\infty)$
and $c=c(\de,q,\alpha )$,
\begin{equation*} \label{eqA1}
\norm{\nb v}_{L^\alpha _tL^q_x(Q_{\delta R})} \le c
\norm{f}_{L^\alpha _tL^q_x(Q_R)} + c R^{-4+
3/q}\norm{v}_{L_t^\alpha L_x^1(Q_R)}.
\end{equation*}
Using $f=v_iv_j$ and the assumption \eqref{vbound2}, the first
integral on the right is bounded by
\[
\norm{v^2}_{L^\alpha _tL^q_x(Q_R)}\le R^{3/q}\bke{\int_{-\infty}^t
  |\tau|^{-\alpha }d\tau}^{1/\alpha }
= R^{3/q} |t|^{1/\alpha  -1}.
\]
The second  term  $R^{-4+
3/q}\norm{v}_{L_t^\alpha L_x^1(Q_R)}$  is bounded by
\[
 R^{-1+
3/q}\bke{R^2 |t|^{-\alpha /2}}^{1/\alpha } = R^{-1+ 3/q+2/\alpha } |t|^{-1/2}
=  R^{3/q} |t|^{1/\alpha  -1}(R^{-2}|t|)^{1/2-1/\alpha }.
\]
These show \eqref{obound}.
\end{proof}

The following is our key lemma.

\begin{lemma} \label{th2-2}
Suppose that the velocity $v$ satisfies the bound \eqref{vbound} and
$\norm{v}_{L^\infty_t L^1_x} \le C_1$ in $Q_1$.  There is $\de \in
(0,1)$ such that, for any small $\e>0$ there is a constant $C_2>0$ so
that $($recall $r=(x_1^2+x_2^2)^{1/2})$
\begin{equation}\label{1.1}
|v(x,t)| \le C_2  r^{-1+2 \e} |t|^{-\e} \quad \text{ in } Q_\de.
\end{equation}
\end{lemma}

\begin{proof}
{\bf Step 1}.  We first bound the second moment of $\om_\th$. Denote
$q=\om_\th$. Its equation can be written as
\begin{equation}\label{q-eq}
\left [\partial_t + b \cdot \nabla - \Delta - \frac {v_r}r \right
  ]\,q + \partial_{z} F = -\frac{q}{r^2} , \qquad F= \frac{-v_\theta^2 }{r}.
\end{equation}
See for instance \cite{CSTY}.  Above the vector $b$ is a part of $v$,
\begin{equation} \label{b-def}
b = v_r e_r + v_z e_z,
\quad
b\cdot \nabla =v_r \partial_r+v_z \partial_z. %, \quad \div b = 0.
\end{equation}
Note that
\begin{equation} \label{b-eq}
\div b = 0,
\quad
\curl b = \om_\th e_\th.
\end{equation}
The first equation for $b$ is because $b=v-v_\th e_\th$,
$\div v = 0$ and $\div (v_\th e_\th)=r^{-1}\pd_\th v_\th=0$.
The second can be read from
\eqref{omegaTHETA}, \eqref{omega-formula} with $v_\th$ replaced by $0$.
The term $\frac{q}{r^2}$ in \eqref{q-eq} has a good sign and will drop out in our
estimates below.  For any $x_0$ fixed with $r_0 >0$, let $\xi (x)$ be
a smooth cutoff function at $x_{0}$ with radius $R=r_{0}/10$. For any
$t$, let $\chi (x, s) = \xi(x) \eta(s)$ where $\eta(t)$ is a smooth
cutoff function so that $\eta (t)=1$ and $\eta(t_0)=0$ with $t_0 = t -
R^2$.  Let $B$ be the characteristic function of the ball centered at
$x_{0}$ with radius $R$ and $\phi (x, s)=B(x) 1(t_0 \le s \le t)$.

 Multiply
\eqref{q-eq} by $\chi^{2} q$ and integrate in $\R^3 \times (t_0,t)$.
We get
\begin{equation}\label{eq2}
\begin{aligned}
\int_{\mathbb{R}^3} \tfrac 12 | \chi  q|^2(t) +
&
\int_{t_0}^t \int_{\mathbb{R}^3} |\nabla(\chi q)|^2
\\
\le &
\int_{t_0}^t \int_{\mathbb{R}^3} \left [ q^2 \bke{ b \chi \cdot
\nabla \chi  + |\nabla \chi |^2 +\frac {\chi^2 v_r}r - \chi \dot \chi }  +  \chi F \partial_{z} (\chi q)
+  \chi F q \partial_{z} \chi \right ],
\end{aligned}
\end{equation}
where $\dot \chi$ is the time derivative of $\chi$.
The last term is bounded by
\[
\int_{t_0}^t\int_{\mathbb{R}^3}
 \chi F q \partial_{z} \chi \le
\int_{t_0}^t\int_{\mathbb{R}^3}
q^2 |\nabla \chi|^2 + \chi^2 F^2 .
\]
The second term on the right hand side of equation \eqref{eq2} can be
bounded by
\[
 \int_{t_0}^t\int_{\mathbb{R}^3} \chi F \partial_{z} (\chi q)
 \le \int_{t_0}^t\int_{\mathbb{R}^3} \left [\frac  {\chi^2 F^2} 2 +
 \frac 1 2 |\nabla  (\chi q)|^2 \right ].
\]
Notice the support of $\chi$ has  a distance at least $R$ from the $z$ axis.
From the assumption on $v_\theta$, we have for any $0\le \e \le 1$
\[
|F \chi|(x,s) \le  C R^{-2+  \e} |s|^{-1/2-\e/2}  \chi(x,s).
\]
Thus we have the integral bound for $\e >0$
\[
\int_{t_0}^t ds \int |F \chi|^2 (x,s) dx \le C R^{-1+ 2 \e} |t|^{-\e}.
\]
Now we can derive the following bound from \eqref{eq2}:
\begin{equation}\label{eq2-1}
\int_{\mathbb{R}^3}  | \chi  q|^2(t)
\le
4 \int_{t_0}^t \int_{\mathbb{R}^3} \left [ q^2  \bke{ |b \chi \cdot
\nabla \chi|  + |\nabla \chi |^2 +\frac {\chi^2 |v_r|}r + |\dot \chi|} \right ]  + C R^{-1+ 2 \e} |t|^{-\e}.
\end{equation}

From the assumption \eqref{vbound}, we also have for $s<t$
\begin{equation*}
|b \chi \cdot
\nabla \chi | \le C \chi  |s|^{-1/2} R^{-1} ,     \quad
|\dot \chi|+  |\nabla \chi |^2  \le C R^{-2} \phi, \quad
 \frac {\chi^2 |v_r|}r  \le C \chi^2 |s|^{-1/2} R^{-1}.
\end{equation*}
Thus we can bound the integral on the right hand side of
\eqref{eq2-1} to get
\begin{equation}\label{eq3}
\int_{\mathbb{R}^3}  | \chi  q|^2  (t)
\le
\int_{t_0}^t  ds  \int_{\mathbb{R}^3}    [   s^{-1/2} R^{-1} +R^{-2} ]   q^2 (s) \phi
+ R^{-1+ 2 \e} |t|^{-\e}.
\end{equation}

We now assume $|t|<R^2$.  Thus $R^{-2} \lec |s|^{-1/2} R^{-1}$ in supp\!
$\phi$ and by Lemma \ref{th2-1},
\begin{equation*} %\label{ext1}
\| q  \phi \|_{L^{3,4}_{t, x}} \le R^{3/4} |t|^{-2/3}.
\end{equation*}
This implies that
\[
\int \int_{\mathbb{R}^3} ds ~ s^{-1/2} R^{-1} q^2 (s) \phi \le R^{-1}
\|s^{-1/2} \phi \|_{L^{3, 2}_{t, x}} \|q^2 \phi \|_{L^{3/2, 2}_{t, x}}
\le R^2 |t|^{-3/2}.
\]
Therefore, from \eqref{eq3} we have
\begin{equation}\label{eq4}
\int_{\mathbb{R}^3} | \chi q|^2 (t) \le R^{2} |t|^{-3/2} .
\end{equation}

Let $\tilde \chi, \tilde B, \tilde \phi$ be the functions similar to
$\chi, B, \phi$ with $R$ replaced by $c R$ for some small constant $c$,
say $c = 1/100$. Clearly, all previous results, in particular
\eqref{eq3}, remain true if we added tildes. We also have
\[
\int_{\mathbb{R}^3} q^2(s) \tilde B \le C \int_{\mathbb{R}^3}  | \chi  q|^2  (s).
\]
We can now use this bound in  \eqref{eq3} (the tilde version) and obtain
\begin{equation}\label{eq2.10}
\int_{\mathbb{R}^3} | \chi q|^2 (t) \le R |t|^{-1} .
\end{equation}

Notice that \eqref{eq2.10} is a better estimate than \eqref{eq4}. We
can repeat this procedure in finite steps to show that,
under the assumption $|t|<R^2$,
\begin{equation}\label{eq5}
\int_{\mathbb{R}^3}  | \chi  q|^2  (t)
\le    R^{-1+ 2 \e} t^{-\e}.
\end{equation}

Assume now $|t|>R^2$. Thus $ |s|^{-1/2} R^{-1} \lec R^{-2} $ in supp\!
$\phi$ and by Lemma \ref{th2-1},
\begin{equation*} %\label{ext1}
\| q  \phi \|_{L^{3,4}_{t, x}} \le R^{5/12} |t|^{-1/2}.
\end{equation*}
We have
\[
\int \int ds R^{-2} q^2 \phi \le R^{-2} \|
\phi \|_{L^{3, 2}_{t, x}} \|q^2 \phi \|_{L^{3/2, 2}_{t, x}}
\le R |t|^{-1}.
\]
Thus
\[
\int |\chi q|^2(t)  \le R t^{-1} + R^{-1+ 2 \e} t^{-\e}
\le  R^{-1+ 2 \e} |t|^{-\e},
\]
which is \eqref{eq5}.
\medskip

{\bf Step 2.}
We now bound the fourth moment of $q$.
Similar to the derivation of \eqref{eq2}, we now have
\begin{equation}\label{eq1}
\begin{aligned}
\int_{\mathbb{R}^3}  | \chi  q^{2}|^2 (t) +
&
 \int_{t_0}^t \int_{\mathbb{R}^3} |\nabla(\chi q^{2})|^2
\\
\le &
\int_{t_0}^t\int_{\mathbb{R}^3} q^4 \bke{ b \chi \cdot
\nabla \chi  + |\nabla \chi |^2 + |\Delta \chi|
+\frac {\chi^2 v_r}r + |\dot \chi| }  +  |\chi q F \partial_{z} (\chi q^{2})|
+ |q^{3} F \chi \partial_{z} \chi| .
\end{aligned}
\end{equation}
From the Schwarz inequality, we have
\[
\chi q F \partial_{z} (\chi q^{2})\le \frac 1 2 \chi^{2} q^{2}F^{2}  + \frac 1 2 |\nabla (\chi q^{2})|^{2}.
\]
\[
 q^{3} F \chi \partial_{z} \chi  \le R^{-2}\chi q^{4}  +\chi R^{2}  F^{4}.
\]
From \eqref{vbound}, we have
\[
\iint \chi R^{2} F^{4} \le \iint \chi R^{-8+2\e} s^{-1-\e} \le
R^{-5+2\e} |t|^{-\e}.
\]
From the bound on $\int \chi^2 q^2$ in \eqref{eq5}
\[
\iint \chi^{2} q^{2}F^{2} \lec R^{-6}\int_{t-R^2}^t R^{-1+2\e} t^{-\e}ds
\le R^{-5+2\e} |t|^{-\e}.
\]
Therefore, we have
\begin{equation}\label{5}
\int_{\mathbb{R}^3}  | \chi  q^{2}|^2  (t)
\le
\int  \int_{\mathbb{R}^3} [   s^{-1/2} R^{-1} +R^{-2} ]    q^4  \phi
+  R^{-5+2 \e} |t|^{-\e}.
\end{equation}

We now assume $|t|<R^2$. Using the bound on $\| \omega_\theta
\|_{L^{6, 8}_{t, x}} \le CR^{3/8}t^{-5/6}$ in \eqref{obound}, we have
\begin{equation}\label{6}
\int_{\mathbb{R}^3}  | \chi  q^{2}|^2  (t)
\le
 R^2 |t|^{-7/2}
+  R^{-5+2 \e} |t|^{-\e}.
\end{equation}
Now plug \eqref{6} into \eqref{5}, we obtain a better result. Repeat this
procedure as in Step 1 until we get
\begin{equation}\label{7}
\int_{\mathbb{R}^3}  | \chi  q^{2}|^2  (t)
\le   R^{-5+2 \e} |t|^{-\e},
\end{equation}
under the assumption $|t|<R^2$.  For the other case $|t|>R^2$, Using
H\"older and the bound $\| \omega_\theta \|_{L^{6, 8}_{t, x}} \le
CR^{-7/24}t^{-1/2}$ from \eqref{obound} to estimate \eqref{5}, we get
\[
\int_{\mathbb{R}^3}  | \chi  q^{2}|^2  (t)
\le
 R^{-1} |t|^{-2}
+  R^{-5+2 \e} |t|^{-\e}\le R^{-5+2 \e} |t|^{-\e},
\]
in one step.

\medskip

{\bf Step 3}. We now prove the pointwise bound \eqref{1.1} for $v$.
Since we have already good estimates for $v_\th$, it suffices to estimate $b$,
which satisfies \eqref{b-def}, \eqref{b-eq} with $\omega_\th=q$.
Let $\de>0$ be a small number so that \eqref{eq5} and \eqref{7} are
valid for $(x_0,t) \in Q_{8 \de}$.  Let $J(x)$ be a smooth cut-off
function for the ball of radius $4\de$,
with $J(x)=1$ for $|x|\le 2\de$.
Define
\[
\al(x)= \int \frac 1{4\pi|x-y|} \curl (Jq e_\th)(y)\,dy
=\int \bke{\nabla_y \frac 1{4\pi|x-y|}} \times (Jq e_\th)(y)\,dy
.
\]
By the vector identity
\begin{equation}\label{vector-id}
-\Delta b = \curl \curl b - \nb \div b,
\end{equation}
the difference $b - \al$ is harmonic in the ball of radius $2\de$ and
hence
\[
\norm{b-\al}_{L^\infty (B_{\de})} \le
\norm{b-\al}_{L^1 (B_{2\de})} \le
\|\alpha \|_{L^1 (B_{2\de})} + \| b\|_{L^1 (B_{2\de})}.
\]
The last term is bounded by order one since $v$ is in $L^\infty_t
L^1_x$. We now estimate $\al$.

For  $x_0\in B_{2\de}$ let $R=\tilde c r_0$ with $\tilde c$
sufficiently small and $B(y)=1(|y-x_0|<R)$. Omitting the
$t$-dependence,
\begin{equation}
|\al |(x_{0}) \le
\int_{\mathbb{R}^3}
 \frac { |Jq({y})| } {|{x_{0}}-{y}|^2} d{y}
 \le \int_{\mathbb{R}^3}  \frac { |q({y})| } {|{x_{0}}-{y}|^2}  B(y) d{y}
+ \int_{\mathbb{R}^3} \frac { |Jq({y})| } {|{x_{0}}-{y}|^2} (1-B(y)) d{y}.
\end{equation}
From the H\"older inequality and \eqref{7}, the first term on the right hand
side is bounded by
\[
\int_{\mathbb{R}^3}  \frac { |q({y})| } {|{x_{0}}-{y}|^2}  B(y) d{y}
\le R^{-1+\e/2} |t|^{-\e/4}.
\]

From the H\"older inequality and \eqref{eq5}, we have the following variation of
\eqref{eq5} for $|x|<8\de$
\begin{equation}\label{eq5c}
 r_{x}^{-3}\int  | 1 (|x-y|\le r_{x}/200)  q (y,t)|  dy
\le    r_{x}^{-2 + \e} |t|^{-\e/2}.
\end{equation}
Multiply by
\[
|x_{0}-x|^{-2}\cdot 1(|x_0-x| \ge R/40)\cdot 1(|x|<8\de),
\]
and integrate over $x$ to have
\[
 \iint   r_{x}^{-3}\frac {1 (|x-y|\le r_{x}/200)  |q (y)|  1(|x_0-x| \ge R/40)} {|x_{0}-x|^{2}}1(|x|<8\de) dx dy,
\]
\[
\le   \int dx  \frac {r_{x}^{-2 + \e} t^{-\e/2} 1(|x_0-x| \ge R/40)}  { |x_{0}-x|^{2}}
\le C R^{-1 + \e} |t|^{-\e/2}.
\]
The left hand side is bounded below by
\[
 \int   r_{x}^{-3}\frac {1 (|x-y|\le r_{x}/200)  |q (y)|  1(|x_0-x| \ge R/40)} {|x_{0}-x|^{2}}1(|x|<8\de) dx dy
 \]
\[
 \ge C \int  dx dy r_{y}^{-3} 1 (|x-y|\le r_{y}/400) |Jq|(y)  \frac { (1-B(y))} {|x_{0}-y|^{2}}
 \ge C \int   dy  |Jq|(y)  \frac { (1-B(y))} {|x_{0}-y|^{2}}.
\]
Above for the first inequality we have used that $y$ is in a small
neighborhood of $x$ for the integrand to be nonzero,
in particular $r_x \sim r_y$ and $|x_0-x|\sim |x_0-y|$.
We have thus proved that
\begin{equation}\label{eq5d}
 \int  |x_{0}-y|^{-2} (1-B (y))  | Jq (y)|  dy
 \le   C R^{-1 + \e} |t|^{-\e/2}.
\end{equation}

Since all $\e>0$ in the proofs are arbitrarily small, this
proves the same bound for $|\al(x_0,t)|$.  It follows that
$\norm{\al(t)}_{L^1_x(B_{2\de})} \le \int_{B_{2\de}} r_x^{-1+\e}
t^{-\e/2}dx \le C t^{-\e/2}$ and we get the pointwise bound for
$b$ in $Q_\de$.
\end{proof}

%%%%%%%%%%%%%%%%%%%%%%%%%%%%%%%%%%%%%%%%%%%%%%%%%%%%%%%%%%%%%%%%%%%%%%%%%%%%
\section{The case $|v| \le C r^{-1+\e} |t|^{-\e /2}$}
%%%%%%%%%%%%%%%%%%%%%%%%%%%%%%%%%%%%%%%%%%%%%%%%%%%%%%%%%%%%%%%%%%%%%%%%%%%%

In this section, we prove the estimate \eqref{assumption1} from the assumption
$|v| \le C r^{-1+\e} |t|^{-\e /2}$. Our main idea is the following Theorem \ref{th3-1}
which states that the space singularity can be replaced by the time singularity.

%\sidenote{$t>0$ inconsistent with Sect 2}
\begin{theorem} \label{th3-1}
Suppose for some $\e\in (0,1/2)$ we have
\begin{equation}\label{*}
|v(x, -t)| \le  C r^{-1+\e} t^{-\e /2},\quad (x,-t)\in Q_1.
\end{equation}
Then for any $\de\in (0,1)$ and  $0 <\alpha < 1/2$, there is a constant
$C$ such that
\begin{equation}\label{2.1}
|v(x, -t) | \le C r^{-2\alpha} t^{-1/2+\alpha}, \quad (x,-t)\in Q_{\de}.
\end{equation}
\end{theorem}

\begin{proof} We shall need the following Lemma
which exchanges the space singularity with
the time singularity by replacing $\e$ with $2  \e$.  The idea of the proof is to view the Navier-Stokes equation
as a linear equation with a source term $ v \cdot \nabla v$. Since this term is in
the form $v^{2}$, we naturally increase the time singularity to $|t|^{-\e}$. The spatial singularity
will come out correctly  due to the scaling invariance of the Navier-Stokes equation. This can be seen easily if we pretend that the kernel of the linear Stokes
equation is a heat kernel. The general case only involves a minor technicality to deal with the divergence free condition.

%This theorem can be shown by iterating the following lemma.

\begin{lemma} \label{th3-2}
Suppose \eqref{*} holds in $Q_1$.  Then for any $\tau \in (0,1)$ there
is a constant $C$ such that
\begin{equation}\label{**}
|v(x, -t)| \le  C r^{-1+2\e} t^{-\e }, \quad (x,-t)\in Q_{\tau}.
\end{equation}
\end{lemma}

Note that \eqref{*}, \eqref{2.1} and \eqref{**} are all invariant
under the natural scaling of \eqref{nse}. Assuming this Lemma, we now finish the proof of
Theorem \ref{th3-1}.

Suppose that \eqref{*} holds for some $\e\in (0,1/2)$. Then from Lemma \ref{th3-2},
we can increase $\e$ by a factor of two. In fact,  \eqref{*} and \eqref{**}  implies that
\begin{equation}\label{***}
|v(x, -t)| \le  C r^{-1+\beta} t^{-\beta /2},\quad (x,-t)\in Q_{\tau},
\end{equation}
for all $\e \le \beta \le 2 \e$. Iterating this procedure, we obtained that
\eqref{2.1} holds for  $ 0< \alpha  < (1-\e)/2$.  It remains to show that
$(1-\e)/2$ can be replaced with $1/2$.

We have shown that \eqref{2.1} holds for small $\alpha$.
Notice that for small $\alpha$ condition \eqref{2.1} is very close to the assumption \eqref{thm-a} (which is the case $\alpha=0$).
One can easily check that all arguments in Section 2
remain valid if the assumption \eqref{thm-a} is replaced by \eqref{2.1} if $ \alpha$ is sufficiently
 small. Then the conclusion of Lemma \ref{th2-2} holds in this case.
 So that we are able to conclude that
\eqref{*} holds  for arbitrarily small $\e$. Iterating this procedure  proves that \eqref{2.1} holds for  $ 0< \alpha < 1/2$.
\end{proof}

\bigskip

To prove the lemma, we write the Navier-Stokes equations \eqref{nse}
as a Stokes system with force
\[
\pd_t v_i - \Delta v_i + \pd_i p = \pd_j f_{ij}, \quad f_{ij} =- v_i
v_j.
\]
Recall key steps in \cite{CSTY}: $v =u + \tilde v$ where $\tilde v$ is
defined as follows: Let $P$ be the Helmholtz projection in $\R^3$,
i.e., $(Pg)_i = g_i - R_i R_k g_k$. Let $\zeta(x,t)\in C^\infty
(\R^4)$, $\zeta \ge 0$, $\zeta = 1$ on $Q_{\frac {1+ \tau}2 }$ and $\zeta=0$
on $\R^3 \times (-\infty,0] - Q_1$. Notice that we cutoff at order
one.  For a fixed $i$, define
\[
\tilde v_i(x,t) = \int_{-1}^t \Gamma(x-y,t-s) \, \partial_j (
F_{ij})(y,s)\,dy\,ds,
\]
where $\Gamma$ is the heat kernel and $ F_{ij} = f_{ij} \zeta - R_i
R_k (f_{kj}\zeta)$.

With this choice of $\tilde v$, $u$ satisfies the homogeneous Stokes
system in $Q_{\frac {1+ \tau}2}$ and the following bounds:
\begin{equation}\label{eqA4}
\norm{\nb u}_{L^s_t L^q_x(Q_{\tau})} \le c\norm{u}_{L^s_tL^1_x(Q_1)}
\le c \norm{v}_{L^s_t L^1_x(Q_1)} + c\norm{\td v}_{L^s_t L^1_x(Q_1)},
\end{equation}
provided that $1<s, q <\infty$.  One can check that the proof in
\cite{CSTY} gives \eqref{eqA4} for $s=\infty$. The requirement $s<\infty$
is for the estimates of $\tilde v$.

\begin{lemma}\label{th3-3}
Under the assumption
\eqref{*}, we have
\begin{equation}\label{eq3.5}
| \tilde v(x_0, -t)| \le C r_0^{-1+2\e} t^{-\e },\quad (x_0,t) \in
\R^3 \times (0,1).
\end{equation}
\end{lemma}

\begin{proof} Denote $R= r_0 $. Notice that assumption
\eqref{*} implies \eqref{**} when $R \ge \sqrt t$. Hence we may
assume that $R \le \sqrt t$.
Let $h=f_{ij} \zeta$ and $K= R_i R_j$. Denote
\begin{align*}
 \xi_1 (x,-t)& = \int_{-1}^{-t}\int |s|^{-2} \exp \left [ - \frac
{|x-\tilde x |^2} {4(-t-s)} \right ] \, h (\tilde x,s)\,d\tilde x\,ds
\\
&= \int^{1-t}_{0}\int s^{-2} \exp \left [ - \frac {|x-\tilde x |^2} {4s}
\right ] \, h (\tilde x,-t-s)\,d\tilde x\,ds,
\end{align*}
\[
 \xi_2 (x,-t) = \int^{1-t}_{0} \int s^{-2} \exp \left [ - \frac
{|x-\tilde x |^2} {4s} \right ] \, K h (\tilde x,-t-s)\,d\tilde x\,ds.
\]
We can bound $\tilde v(x,-t)$ pointwisely by $\sum_{i,j,k} |\xi_k(x,-t)|$.
Thus it suffices to show $\sum |\xi_k(x_0,t)|\le C R^{-1+2\e} t^{-\e }$.
We shall only bound $\xi_2$ since the bound for
$\xi_1$ is identical. For $x_0$ fixed, let
\[
g (x, s) = \exp \left [   - \frac {|x_0-x |^2} {4s} \right ] s^{-3/2}.
\]
Since $K$ is symmetric, we have
\[
 \xi_2 (x_0,-t) =
\int^{1-t}_{0}   s^{-1/2}  \int d x \, (Kg)( x, s)  \,  h (x,-t-s)\,ds.
\]
%Notice the kernel is $s^{-1/2}$ times the three dim heat kernel.
%Without loss of generality, we shall assume that $z=0$.

Let $h_1(x, a)= 1(r \ge R) h(x, a)$ and $h_2(x, a)= 1(r \le R)h(x,
a)$. Then we have
\[
\int \,d x        \,   (K g)(x, s) h_1(x, -t-s)
\le   \left [ \int \,d x (K g)^p (x, s) |x-x_0|^p \right ]^{1/p}
\left [ \int \,d x h_1^q(x, -t-s) |x-x_0|^{-q} \right ]^{1/q}.
\]
Recall $|x|^a$ is an $A_p$ weight in $\R^n$  provided that
\[
-n< a < n (p -1).
\]
Thus for
\[
0\le p < 3 p -3,
\]
we have
\[
\left [ \int \,d x (K g)^p (x, s) |x-x_0|^p \right ]^{1/p}
\le \left [ \int \,d x  g^p (x, s) |x-x_0|^p \right ]^{1/p}
\le |s|^{-1+3/(2p)} = |s|^{-3/(2q) + 1/2}.
\]
Since $h_1$ is supported in $r \ge R$, we have for
\[
 (3-2 \e)  q > 3, \quad q>1,
\]
the following inequality:
\[
\left [ \int \,d x h_1^q(x, -t-s) |x-x_0|^{-q} \right ]^{1/q}
\le t^{-\e} \left [ \int_{r \ge R} \,d x r^{(-2+ 2 \e)q }  |x-x_0|^{-q} \right ]^{1/q}
\le t^{-\e} R^{-3+ 2 \e + 3/q},
\]
where we have used \eqref{*} in the first inequality.
Thus  we have
\[
 s^{-1/2} \int \,d x        \,   (K g)(x, s) h_1(x, -t-s)
\le  t^{-\e} R^{-3+ 2 \e} \left ( \frac {R^2} s \right )^{3/(2q)}.
\]
Therefore, we have
\[
\int^{R^2}_{0}   s^{-1/2}  \int d x \, (Kg)( x, s)  \,  h_1 (x,-t-s)\,ds
\le C t^{-\e}  R^{-3+ 2 \e}  \int^{R^2}_{0} ds   \left ( \frac {R^2} s \right )^{3/(2q)}
\le C t^{-\e}    R^{-(1- 2 \e) },
\]
provided that
\begin{equation}\label{q1a}
3/(2q) < 1, \quad  (3-2 \e)  q > 3, \quad q>1, \quad 0\le p < 3 p -3.
\end{equation}
For any  $0< \e < 1/2$ fixed, we can solve the last condition  by
\begin{equation}\label{q1a-sol}
q= (3/2)^+.
\end{equation}
Similarly,
\[
\int_{R^2}^{1}   s^{-1/2}  \int d x \, (Kg)( x, s)  \,  h_1 (x,-t-s)\,ds
\le C t^{-\e}  R^{-3+ 2 \e}  \int_{R^2}^{1}  ds \left ( \frac {R^2} s \right )^{3/(2q)}
\le C t^{-\e}    R^{-(1- 2 \e) },
\]
provided that
\begin{equation}\label{q1b}
3/(2q) > 1, \quad  (3-2 \e)  q > 3, \quad q>1, \quad 0\le p < 3 p -3.
\end{equation}
For any  $0< \e < 1/2$ fixed, we can solve the last condition  by
\begin{equation}\label{q1b-sol}
q= (3/2)^-.
\end{equation}

\bigskip
For any $m> 0$ and  $a, b$ dual we have
\[
 \int \,d x        \,   (K g)(x, s) h_2(x, s)
\le  \left [ \int \,d x (K g)^a (x, s) |x-x_0|^{ma} \right ]^{1/a}
\left [ \int \,d x h_2^b(x, -t-s) |x-x_0|^{-mb} \right ]^{1/b}.
\]
If
\[
m a < 3 a -3,
\]
then $|x|^{ma}$ is an $A_a$ weight. Thus
we have
\[
\left [ \int \,d x (K g)^a (x, s) |x-x_0|^{ma} \right ]^{1/a}
\le \left [ \int \,d x  g^a (x, s) |x-x_0|^{ma} \right ]^{1/a}
\le s^{-3/(2b) + m/2}.
\]
We can estimate the last integral by
\[
\int \,d x h_2^b(x, -t-s) |x-x_0|^{-mb}
\le t^{-\e b} \int \,d x r^{-(2-2\e) b} 1(r \le R) |x-x_0|^{-mb}
\]
\[
=
t^{-\e b}   R^{-b(2+m- 2 \e) + 3} \int \,d x r^{-(2-2\e) b} 1(r \le 1)|x-(x_0/R)|^{-mb}
\le C t^{-\e b}   R^{-b(2+m- 2 \e) + 3},
\]
where the equality is due to scaling and we have assumed that
\[
(1-\e) b < 1, \quad 1< mb <3.
\]
Therefore, we have
\[
\int^{R^2}_{0}   s^{-1/2}  \int d x \, (Kg)( x, s)  \,  h_2 (x,-t-s)\,ds
\le C t^{-\e}    R^{-(3- 2 \e) } \int^{R^2}_{0} ds  \left ( \frac  {R^2} s \right )^{3/(2b) -m/2+1/2}
\le C t^{-\e}    R^{-(1- 2 \e) },
\]
provided that
\begin{equation}\label{a1a}
b > \frac 3 {m+1} , \quad (1-\e) b < 1, \quad 1< mb <3 , \quad m a < 3 a -3.
\end{equation}
Since $a$ and $b$ are dual, the last condition is equivalent to $m < 3/b$, which is part of the third condition.
It is easy to check that the following condition implies \eqref{a1a}
\begin{equation}\label{a2b}
1/2 < m,  \quad  \frac 3 {m+1} < b < \frac 3 m , \quad b < \frac 1  {1-\e}.
\end{equation}
For any  $0< \e < 1/2$ fixed, we can solve the last condition  by
\begin{equation}\label{a1a-sol}
b= \frac 1  {1-\e} - \mu, \quad m = 3-  3 \e,
\end{equation}
for $\mu$ small enough.

%For any  $0< \e < 1/2$ fixed, we can solve the last condition  by
%\begin{equation}\label{a1a-sol}
%b= 1+\mu,  \quad m = 2-\mu
%\end{equation}
%for $\mu$ small enough.
Similarly, we have
\[
\int_{R^2}^{1}   s^{-1/2}  \int d x \, (Kg)( x, s)  \,  h_2 (x,-t-s)\,ds
\le C t^{-\e}    R^{-(3- 2 \e) } \int_{R^2}^{1} ds  \left ( \frac  {R^2} s \right )^{3/(2b) -m/2+1/2}
\le C t^{-\e}    R^{-(1- 2 \e) },
\]
\begin{equation}\label{a1b}
b < \frac 3 {m+1} , \quad (1-\e) b < 1, \quad 1< mb <3 , \quad m a < 3 a -3.
\end{equation}
It is easy to check that the following condition implies \eqref{a1b}
\begin{equation}%\label{a2b}
 \frac 1 m  < b < \frac 3 {m+1} , \quad b < \frac 1  {1-\e}.
\end{equation}
This equation has a solution provided that there is an $m$ solving  the equation
\[
1 - \e < m < 2-3 \e.
\]
This is clearly so for any $0< \e < 1/2$ fixed.
%For any $0< \e < 1/2$ fixed, we can solve the last condition by
%\begin{equation}\label{a1b-sol}
%b= 1+\mu,  \quad m = 2-3 \mu
%\end{equation}
%for $\mu$ small enough.
We have thus  proved the Lemma.
\end{proof}

\bigskip

We now conclude the proof of Lemma \ref{th3-2}. Let $\ga= \frac {1-
\tau}4$ and denote $Q_1^t = B_1(0)\times (-1,-t)$.  For any
$(x_0,t)\in Q_\tau$, $Q_{2\ga}(x_0,t) \subset Q_{ \frac {1+\tau}2}
\cap Q_1^t$ and hence $u$ satisfies the homogeneous Stokes system in
$Q_{2\ga}(x_0,t)$. By \eqref{eqA4}, assumption \eqref{*} and Lemma
\ref{th3-3}, we have
\begin{equation*}%\label{c1}
\norm{\nb u}_{L^\infty_t L^4_x(Q_\ga(x_0,t))}
\le c\norm{u}_{L^\infty_tL^1_x(Q_{2\ga}(x_0,t))}
\le c \norm{v}_{L^\infty_t L^1_x(Q_1^t)} + c\norm{\td v}_{L^\infty_t L^1_x(Q_1^t)}
\le  c t^{-\e/2} + c t^{-\e} \le c t^{-\e}.
\end{equation*}
By the Sobolev inequality,
\begin{equation*}
\norm{u}_{L^\infty_t L^\infty_x(Q_\ga(x_0,t))}
\le c\norm{\nb u}_{L^\infty_t L^4_x(Q_\ga(x_0,t))}
+ c\norm{u}_{ L^\infty_t L^1_x(Q_{2\ga}(x_0,t))} \le  c t^{-\e}.
\end{equation*}
Together with \eqref{eq3.5}, we have thus proved Lemma \ref{th3-2}.
\hfill $\square$

\section{Local Maximum Estimate}

In sections 2 and 3 we have proved the bound \eqref{assumption1} under
both assumptions \eqref{thm-a} and \eqref{thm-b} with $\e>0$.  Our
goal in the remaining sections 4, 5, and 6 is to show that the proof
in the paper \cite{CSTY} can be extended in this case.  This section
proves local maximum estimates assuming \eqref{assumption1}.  These
estimates will be used to obtain H\"older continuity of $rv_\theta$ in
section 5 and to bound $\Omega=\bar \om_\th /r$ of the limit solution
in section 6.

Suppose $u$ is the smooth function satisfying
\begin{equation}\label{u2}
\partial_t u - L^* u = 0, \qquad
L=  \Delta + {\frac 2r} \pd_r -b \cdot \nb ,
\end{equation}
We now derive parabolic De Giorgi type energy estimates for this equation under the assumption
\begin{equation}\label{assumption}
|b| \le   C_* r^{-1+2\e}|t|^{-\e }.
\end{equation}
Above $C_*>0$ is an absolute constant which is allowed to be large, above
$\epsilon>0$ is sufficiently small.

Consider a  test function $0\le \zeta_1(x,t)\le 1$ defined on $Q_1$ for which $\zeta_1=0$ on
$\pd B_1 \times [- 1^2,0]$ and $\zeta_1 =1$ on $Q_{\sigma }$ for $0<\sigma<1$.  Suppose that $\zeta_1(x, -1)=0$.
Now consider the rescaled test function $\zeta(x,t)=\zeta_1(x/R,  t/R^2)$
on $Q_R$.
Define $(u)_\pm = \max\{ \pm u, 0\}$ for a scalar function $u$.
Multiply \eqref{u2} by $p (u-k)_\pm^{p-1} \zeta^2$ for  $1<p \le 2$ and $k\ge 0$ to obtain
\begin{gather*}
\left.\int_{B_R} \zeta^2 (u-k)_\pm^p\right|^t_{- R^2}
+
\frac {4(p-1)}p \int_{-R^2} ^t dt' \int_{B_R} dx |\nabla
((u-k)_\pm^{p/2}\zeta)|^2
\\
= 2\int_{-R^2} ^t dt' \int_{B_R} dx   (u-k)_\pm^p \left( \zeta
\frac{\partial \zeta}{\partial t}+|\nabla \zeta|^2 + \frac{2-p}p
\zeta\Delta \zeta -2\zeta \frac{\partial_r \zeta}{r} +\bar{b}
\cdot \zeta \nabla \zeta\right)
\\
-
 2\left.\int_{-R^2} ^t dt'\int_{B_R} dz ~ \zeta^2 (u-k)_\pm^p\right|_{r=0}.
\end{gather*}
Notice that the last term is negative.

Let $v_\pm\equiv (u-k)_\pm^{p/2}$.
To estimate the term involving $b$ we use  Young's inequality
$$
\int_{\mathbb{R}^3} v_\pm^2 b  \zeta\cdot \nabla \zeta
\le \delta \frac{R^{-1+\alpha}}{1+\alpha } \int_{\mathbb{R}^3} v_\pm^2 \zeta^2 \left |b \left (\frac t {R^2} \right )^\e \right |^{1+\alpha} +
C_\delta\frac{\alpha  R^{-2+ (1+\alpha)/\alpha}}{1+\alpha }
 \int_{\mathbb{R}^3} v_\pm^2  \zeta^2  \left [   \left ( \frac {R^{2} } t \right )^\e \frac { |\nabla \zeta|} { \zeta} \right ]^{ (1 + \alpha)/\alpha}.
$$
This holds  for small $\delta>0$ and $\alpha >0$ to be chosen.
Further choose $\zeta$ to decay like  $( 1-|x|/R)^n $  near the boundary of $B_R$.
If $n$ is large enough
(depending on $\alpha$)  we have
$$
C_\delta\frac{\alpha  R^{-2+ (1+\alpha)/\alpha}}{1+\alpha }
 \int_{\mathbb{R}^3} v_\pm^2  \zeta^2  \left [ \left ( \frac {R^{2} } t \right )^\e  \frac {|\nabla \zeta|} \zeta \right ]^{ (1 + \alpha)/\alpha}
 \le
 C
R^{-2} \left ( \frac {R^{2} } t \right )^{ \e (1 + \alpha)/\alpha}  \int_{B_R}  v_\pm^2.
$$
We also use the H{\"o}lder and Sobolev inequalities to obtain
\begin{equation*}
\begin{split}
\delta \frac{R^{-1+\alpha}}{1+\alpha } \int_{\mathbb{R}^3} v_\pm^2 \zeta^2
\left |b \left (\frac t {R^2} \right )^\e \right |^{1+\alpha}
&\le
\delta\left(R^{(-1+\alpha)3/2}  \int_{B_R}  \left |b \left (\frac t {R^2} \right )^\e \right |^{(1+\alpha)3/2}\right)^{2/3} \int_{\mathbb{R}^3} |\nabla (v_\pm\zeta ) |^2
\\
&\le
\delta C\int_{\mathbb{R}^3} |\nabla (v_\pm\zeta ) |^2.
\end{split}
\end{equation*}
For  $b$ satisfies \eqref{assumption}, there is an
$\alpha$ small enough so that the last inequality holds.  We conclude that
\begin{equation}
\label{b-term}
\begin{split}
\int_{\mathbb{R}^3} v_\pm^2 b  \zeta\cdot \nabla \zeta
&\le
\delta C\int_{\mathbb{R}^3} |\nabla (v_\pm\zeta ) |^2
+
 C
R^{-2}   \int_{B_R}  \left ( \frac {R^{2} } t \right )^{ \e (1 + \alpha)/\alpha} v_\pm^2.
\end{split}
\end{equation}

We have $\partial_r \zeta /r = \partial_\rho \zeta
 /\rho$ since $\zeta$ is radial;  so that
the singularity $1/\rho$ is effectively $1/R$.  We thus have
\begin{equation}\label{deGiorgiPm}
\begin{aligned}
\sup_{- \sigma^2R^2<t<0} \int_{B_{\sigma R}\times\{t\}}  |( u -k)_\pm|^p  &+
\int_{Q_{\sigma R}} |\nabla  ( u -k)_\pm^{p/2}  |^2
\\
&\le
\frac {C_*}{(1-\sigma)^2 R^2}  \int_{Q_{R}} \left ( \frac {R^{2} } t \right )^{ \e (1 + \alpha)/\alpha}  |( u -k)_\pm|^p.
\end{aligned}
\end{equation}
Our goal will be to establish $L^p$ to $L^\infty$ bounds for functions in this energy class.

The estimates in this section will be proven for a general function $u=\Omega$ satisfying \eqref{deGiorgiPm}:

\begin{lemma}\label{weakH}  Suppose $u$ satisfies \eqref{deGiorgiPm} for $1<p\le 2$ with $\epsilon>0$ sufficiently small.  Then for $0<R\le 1$ we have the estimate
$$
\sup_{Q_{R/2}} u_{\pm} \le C(p, C_*)
\left(R^{-3-2}\int_{Q_{R}}|u_\pm|^p \right)^{1/p}.
$$
\end{lemma}

\begin{proof}
For $K>0$ to be determined and $N$ a positive integer we define
\begin{gather*}
k_N
=
k_N^\pm=(1\mp 2^{-N})K,
~R_N=(1+2^{-N})R/2,
~\rho_N=\frac{R}{2^{N+3}},
\\
R_{N+1}<\bar{R}_N=(R_N+R_{N+1})/2 <R_N.
\end{gather*}
Notice that
$$
R_N-\bar{R}_N
=(R_N-R_{N+1})/2
=(2^{-N}-2^{-N-1})R/4
=\rho_N.
$$
Define $Q_N=Q(R_N)$ and $\bar{Q}_N=Q(\bar{R}_N)\subset Q_N$.
Choose a smooth test function $\zeta_N$ satisfying $\zeta_N\equiv 1$ on $\bar{Q}_N$, $\zeta\equiv 0$ outside $Q_N$ and vanishing on it's spatial boundary, $0\le \zeta_N\le 1$ and
$\left| \nabla \zeta_N\right|\le \rho_N^{-1}$ in $Q_N$.
Further let
$$
A^\pm(N)=\{X\in Q_N: \pm(u-k_{N+1})(X)>0\}.
$$
And $A_{N,\pm}=\left|A^\pm(N) \right|$.
Let $v_\pm = \zeta_N (u-k_{N+1})_\pm^{p/2}$.

Let $\gamma $ be a positive constant (to be chosen) such that $\gamma-1>0$ is very small. H{\"o}lder's inequality yields
\begin{gather*}
\begin{split}
\int_{Q_{N+1}} ~ |(u-k_{N+1})_\pm|^p
&\le
\int_{\bar{Q}_{N}} ~ |v_\pm|^2
\\
&\le
\left(\int_{\bar{Q}_{N}} ~ |v_\pm|^{2(n+2)/(n\gamma)}\right)^{n \gamma /(n+2)}
A_{N,\pm}^{(2 + n (1-\gamma) )/(n+2)}.
\end{split}
\end{gather*}
We will use the following  Sobolev inequality which holds for functions
vanishing on $\partial B_R$:
\begin{equation*}
\int_{Q_R} |u|^{2(n+2)/n}
\le
C(n) \left( \sup_{-R^2<t<0} \int_{B_R\times\{t\}} |u|^2 \right)^{2/n}
\int_{Q_R}  |\nabla u|^2.
\end{equation*}
See \cite[Theorem 6.11, p.112]{MR1465184}.  Above and below $n$ is the spatial dimension, so that $n=3$.
%We are interested in the form
%\begin{equation*}
%\int_{Q_R} |u^{p/2}|^{2(n+2)/n}
%\le
%C(n) \left( \sup_{-R^2<t<0} \int_{B_R\times\{t\}} |u|^p \right)^{2/n}
%\int_{Q_R}  |\nabla u^{p/2}|^2.
%\end{equation*}

Since $v_\pm$ vanishes on the spatial boundary of $Q_N$ we have
\begin{gather*}
\left(\int_{\bar{Q}_{N}} ~ |v_\pm|^{2(n+2)/(n\gamma)}\right)^{n \gamma /(n+2)}
\le
\left(\int_{Q_{N}} ~ |v_\pm|^{2(n+2)/(n\gamma)}\right)^{n \gamma /(n+2)}
\\
\le
C \left [\left( \sup_{-R_N^2<t<0} \int_{B(R_N)\times\{t\}} |v_\pm|^{2/\gamma} \right)^{2/(n+2)}
\left(\int_{Q_{N}}  |\nabla (v_\pm)^{1/\gamma}|^{2}\right)^{n /(n+2)} \right ]^\gamma.
\end{gather*}
We use Young's inequality to bound this above by
\begin{gather*}
\le
C  \left [ \left( \sup_{-R_N^2<t<0}\int_{B(R_N)\times\{t\}}|v_\pm|^{2/\gamma}
+\int_{Q_N}|\nabla (v_\pm)^{1/\gamma}|^{2}
\right) \right ]^\gamma.
\end{gather*}
From \eqref{deGiorgiPm} the above is bounded as
\begin{gather*}
\le
C  \left [ \left( \sup_{-R_N^2<t<0}\int_{B(R_N)\times\{t\}}|(u-k_{N+1})_\pm|^{p/\gamma}
+\int_{Q_N}|\nabla (u-k_{N+1})_\pm^{p/(2\gamma)}|^{2}
\right)
+
\frac{C}{\rho_N^2}\int_{Q_N}|(u-k_{N+1})_\pm|^{p/\gamma} \right ]^\gamma
\\
\le
\left [ \frac{C}{\rho_N^2}\int_{Q_N} \left [ 1+  \left ( \frac {R^{2} } t \right )^{ \e (1 + \alpha)/\alpha}  \right ] |(u-k_{N+1}) _\pm|^{p/\gamma}\right ]^\gamma
\\
\le
\left \{
\frac{C}{\rho_N^2}\int_{Q_N}|(u-k_{N}) _\pm|^{p}
\right \}
\left \{
\frac{C}{\rho_N^2}\int_{Q_N}  \left [ 1+  \left ( \frac {R^{2} } t \right )^{ \e (1 + \alpha)/\alpha}  \right ]^\beta
\right \}^{\gamma/\beta}.
\end{gather*}
Above $\beta$ is the dual exponent of $\gamma$, i.e.
$
\frac 1 \gamma + \frac 1 \beta = 1.
$
For the upper bound above to be finite we require that
 $\e < \alpha/[\beta(1+\alpha)]$. Since $\beta \ge 1$ we have
$$
\int_{Q_N}  \left [ 1+  \left ( \frac {R^{2} } t \right )^{ \e (1 + \alpha)/\alpha}  \right ]^\beta
\le
C\int_{Q_N}  \left [ 1+  \left ( \frac {R^{2} } t \right )^{ \e (1 + \alpha)\beta/\alpha}  \right ]
\le
CR_N^5.
$$
From here our next upper bound is
\begin{gather*}
\le
C\left(\frac{R_N^5}{\rho_N^2}\right)^{\gamma/\beta}
\left \{
\frac{1}{\rho_N^2}\int_{Q_N}|(u-k_{N}) _\pm|^{p}
\right \}
\le
C\left(16 R^3 2^{2N}\right)^{\gamma/\beta}
\left \{
\frac{1}{\rho_N^2}\int_{Q_N}|(u-k_{N}) _\pm|^{p}
\right \}.
\end{gather*}
Further assume $K^p \ge R^{-n-2}\int_{Q(R)} | u_\pm|^p$.  Now define
$$
Y_N \equiv K^{-p} R^{-n-2}\int_{Q_N} | (u-k_{N})_\pm|^p.
$$
Since $k_N^\pm$ are increasing for $+$ or decreasing for $-$ and $Q_N$ are decreasing, $Y_N$ is decreasing.

Chebyshev's inequality tells us that
\begin{gather*}
A_{N,\pm} =\left|\{ Q_N: \pm(u-k^{\pm}_{N+1})>0  \}\right|
=\left|\{ Q_N: \pm(u-k^{\pm}_{N})>\pm (k^{\pm}_{N+1}- k^{\pm}_N)  \}\right|
\\
=\left|\{ Q_N: \pm(u-k_{N})> K/2^{N+1}  \}\right|
\le  2^{p(N+1)} R^{n+2}Y_N.
\end{gather*}
Putting all of this together yields
\begin{gather*}
\int_{Q_{N+1}} ~ |(u-k_{N+1})_\pm|^p
\le   \left(\int_{\bar{Q}_{N}} ~ |v_\pm|^{2(n+2)/(n\gamma)}\right)^{n \gamma /(n+2)}
A_{N,\pm}^{(2 + n (1-\gamma) )/(n+2)}
\\
\le
C\left(\frac{1}{\rho_N^2}\int_{Q_N}|(u-k_{N}) _\pm|^{p}\right) \
R^{3\gamma/\beta} 2^{ 2N \gamma/\beta}
\left(2^{p(N+1)} R^{n+2}Y_N\right)^{(2 + n (1-\gamma) )/(n+2)}
\\
\le
C\left( \frac{1}{\rho_N^2} K^pR^{n+2} Y_N\right)
R^{3\gamma/\beta} 2^{ 2N \gamma/\beta}
\left(2^{p(N+1)} R^{n+2}Y_N\right)^{(2 + n (1-\gamma) )/(n+2)}
\\
\le
C(C_* n) K^pR^{n+2} 2^{ q N}
Y_N^{1+(2 + n (1-\gamma) )/(n+2)}.
\end{gather*}
We have just used $R\le 1$.
Also the exponent is given by
\[
q= \frac {[2 + n (1-\gamma) ]p } {n+2} + \frac  {2\gamma}\beta + 2.
\]
We have thus shown
\begin{equation}\label{Y}
Y_{N+1}
\le C(C_*, n)  2^{ q N} Y_N^{1+(2 + n (1-\gamma) )/(n+2)}.
\end{equation}
We now choose $\gamma > 1$ such that the exponent of $Y_N$ is larger than one:
$
2 + n (1-\gamma) > 0.
$

One can check that if $\kappa$ is large enough, then the following identity will be preserved by \eqref{Y}:
\[
Y_N \le 2^{- \kappa N}.
\]
We are still free to choose $K$ large enough such that the following initial condition holds:
$
Y_1 \le  2^{-\kappa}.
$
\end{proof}

\bigskip

\section{H\"older Continuity}

In this section we prove H\"older continuity of the function
$\Gamma=rv_\theta$ at $t=0$ under the assumption
\eqref{assumption1}. Earlier than $t=0$,  the function
$\Gamma$ is smooth.   Additionally $\Gamma$ satisfies
\begin{equation}\label{Leq}
\frac{d \Gamma}{dt} -L \Gamma= 0, \qquad
 L  =
\Delta
-\frac{2}{r}\frac{\partial }{\partial r}
 - b\cdot \nabla.
\end{equation}
Notice that $\Gamma(r=0, t) = 0$  for all $ -1 \le t < 0$. One can check, using this condition, that both \eqref{deGiorgiPm} and Lemma \ref{weakH} hold.  Together with \eqref{assumption} we then have
\[
\sup_{-1 \le t < 0} \| \Gamma(t) \|_{L^\infty(B_R)} \le C<\infty.
\]
Our argument makes use Nash's fundamental idea for a lower bound (Lemma \ref{NashlowL}).  We consider this interesting in particular because the lower bound is obtained for a {\it solution} directly rather than the usual lower bound for a {\it fundamental solution}.

\bigskip

\subsection{Preliminary Bounds}

Let $X=(x,t)$.
Define the modified parabolic cylinder at the origin
$$
Q(R,\tau)=\{X: |x|<R,  - \tau R^2<t< 0\}.
$$
Here $R>0$ and $\tau \in (0,1]$.
We sometimes for brevity write
$Q_R=Q(R)=Q(R,1)$.
Let
\[
m_2 \equiv \inf_{Q(2R)} \Gamma,
\quad M_2 \equiv \sup_{Q(2R)} \Gamma,
\quad M \equiv M_2 - m_2>0.
\]
Notice that $m_2 \le 0 \le M_2$ since $\Gamma |_{r=0}=0$.

Define
\begin{equation}\label{udef}
u
\equiv
\left\{
\begin{aligned}
2(\Gamma - m_2)/M & \quad \text{if}\quad -m_2 > M_2,
\\
2(M_2 - \Gamma)/M & \quad \text{else.}
\end{aligned}
\right.
\end{equation}
In either case
\begin{equation}\label{ua}
0 \le u(x,t) \le 2, \qquad
a \equiv u|_{r=0}\ge 1,
\end{equation}
and $u$ solves  the equation \eqref{Leq}.

Since  now $u$ is nonnegative, we can make $u$ positive by adding arbitrary small constant
to $u$. This part of argument is standard and from now on we assume that $u >0$.

\subsection{Lower bound on $\|u\|_p$}

Our goal in this section is to prove that there is a lower bound on $u$ with a more general assumption than that was used in our previous paper\cite{CSTY}.  The bound that we prove in this section will serve as an input for Nash's argument as we shall describe it later on. Actually, we only need
a lower bound on $\|u\|_p$ for some $0<p<1$.

Consider the following probability  measure on $Q_R$
\begin{equation}\label{omega}
d\omega = R^{-2}dt R^{-3} dx.
\end{equation}
Define the norm
$$
\| f\|_{L_{t,x}^{q,p}(\omega)}
:=
\left(\int_{-R^2}^0 \frac{dt}{R^2}~ \left( \int_{B_R} \frac{dx}{R^3} ~|f|^p \right)^{q/p} \right)^{1/q}.
$$
We will sometimes write $L^{p}(\omega) = L_{t,x}^{p,p}(\omega)$.
Our main result in this section is the following

\begin{lemma}  Suppose $u$ solves \eqref{Leq} and satisfies condition \eqref{ua}.
Assume that for some small $\beta > 1$ we have
\begin{equation}\label{bassumption}
\|b \|_{L^{3/2, \beta}_{t,x}(\omega)} \le C R^{-1}.
\end{equation}
Then for arbitrary $p\in (0,1)$ the following holds:
\begin{equation}\label{lb}
a\le  C  \|u \|^{p/\alpha}_{L^p(\omega)}.
\end{equation}
Above $\alpha$ is the dual exponent to $\beta$.
\end{lemma}

Notice that  for $b$  satisfying \eqref{assumption} with  any $0\le \epsilon \le 1/2$ fixed, there is a $\beta$ such that the condition \eqref{bassumption} is satisfied.   The following proof is a small modification of the proof
in \cite{CSTY}.

\bigskip
\begin{proof}  We  test the equation with
$pu^{p-1} \zeta^2$ for $0<p<1$ and $\zeta \ge 0$ to have
\[
 -I_7 =\sum_{j=2}^6 I_j.
\]
where
\begin{equation*}
\begin{aligned}
\int_{Q(R)} pu^{p-1} \zeta^2 \frac{\pd u}{\partial t}
= &
\bkt{\int_{B_R} \zeta^2 u^p}^0_{t_1} - \int_{Q(R)} u^p 2 \zeta \frac{\pd \zeta}{\partial t}
\equiv I_1 + I_2,
\\
\int_{Q(R)} pu^{p-1} \zeta^2 (-\Delta u)
=& \frac {4(p-1)}p \int_{Q(R)} |\nabla
(u^{p/2}\zeta)|^2
\\
&+ \int_{Q(R)} 2 u^p \bkt{-|\nabla \zeta|^2 + \frac
{p-2}p \zeta\Delta \zeta}
\equiv I_3 + I_4,
\\
\int_{Q(R)}  pu^{p-1} \zeta^2 b \cdot \nabla u
= &
- \int_{Q(R)}  2u^p b  \cdot \zeta \nabla \zeta
\equiv I_5,
\\
\int_{Q(R)} pu^{p-1} \zeta^2 \frac 2r \pd_r u
= & - \int_{Q(R)} 4 u^p\zeta \zeta _\rho
/\rho - \int_{- R^2}^{0} dt\int_{\mathbb{R}} dz ~ 2 (\zeta^2 u^p)|_{r=0}
\\
\equiv I_6 + I_7.
\end{aligned}
\end{equation*}
For arbitrary $p\in (0,1)$, we see that $I_3$ and $I_7$ are both
non-positive.

Recall $\rho=|x|$.  We choose $\zeta = \zeta_1(\rho) \zeta_2(t)$ where
$\zeta_1(\rho) =1$ in $B_{R/2}$ and $\zeta_1(\rho)$ has compact support
in $B_R$; also $\zeta_2(t) = 1$ if
$t\in  (-\frac{7}{8}R^2, -\frac{1}{8}R^2)$ and $\zeta_2(t)$ has compact support in
$(-R^2, 0)$.
 Thus $I_1=0$ and we have
\[
\frac{6}{4} R^3 a^p \le -I_7 =\sum_{j=2}^6 I_j.
\]
Clearly,
\[
I_2 \le \frac {C}{ R^2} \int_{Q(R)} u^p, \quad I_3 \le 0,\quad  I_4+ I_6
\le \frac {C}{R^2} \int_{Q(R)} u^p .
\]

For any dual $\alpha, \beta$, we now bound $I_5$:
$$
|I_5| \le  R^5 R^{-1}\|b \|_{L^{3/2, \beta}_{t,x}(\omega)}
\|u^p\|_{L^{3, \alpha}_{t,x}(\omega)}.
$$
We thus have for any $\alpha \ge 3$
\[
a^p \le \frac {C }{ R^5} \int_{Q(R)} u^p
+  \|u^p\|_{L^{3, \alpha}_{t,x}(\omega)}
\le  \|u^p\|_{L^{\alpha}_{t,x}(\omega)} \le \|u\|_{L^{\alpha p}(\omega)}^p.
\]
Since $p$ is arbitrary positive number less than one, this proves the Lemma.
Notice that we only use $I_3 \le 0$ in this case.

\end{proof}

\subsection{Nash inequality}

In this section we prove a Nash inequality.
We shall need this inequality in the next section when we prove a Nash lower bound estimate.

Consider a function $f$, which satisfies the bounds $0 \le f  \le M$ for some $M \ge 1$.  Let $\mu$ be a probability measure.
Now consider the average
$$
\alpha= \int \log f ~d \mu.
$$
And define  $g = \log f -\alpha$.  We have the following inequality
\begin{equation}
\frac{\|f\|_1}{M}  \bigg |  \alpha -  \log \int f d \mu  \bigg  |
\le   \|g\|_2.
\label{nashINEQ}
\end{equation}
Above we are using the following definition
$$
\|g\|_p := \left(\int |g|^p~d\mu\right)^{1/p}.
$$
In the rest of this section we will give a short proof.

\bigskip

\begin{proof}  For $0\le \beta\le 1$ we have
$$
 \partial_\beta \log \int e^{\beta g} d \mu
=
\frac {\int g  e^{\beta g} d \mu }  {\int e^{\beta g} d \mu }
=
\frac {\int g f^{\beta } d \mu }  {\int f^{\beta } d \mu }
\le
\frac {\| g \|_2  \|  f^{\beta} \|_2 }  {\| f^{\beta} \|_1 }
\le
\frac {\| g \|_2  M^\beta }  {\| f^{\beta} \|_1 }.
$$
Additionally, since $f$ is bounded, we have
$$
\int f^\beta d \mu =  M^\beta \int  \left ( \frac f M \right ) ^\beta d \mu
\ge  M^\beta \int  \left ( \frac f M \right )  d \mu
=  M^{\beta-1} \| f \|_1.
$$
We conclude
$$
\partial_\beta \log \int e^{\beta g} d \mu
\le   \frac {M \| g \|_2   }  {\| f \|_1 }.
$$
This is all we need.  Now integrate the above with respect to $\beta$ from $0$ to $1$ to obtain

$$
\log \int e^{g} d \mu=
\log \int f d \mu - \alpha
\le   \frac {M \| g \|_2   }  {\| f \|_1 }.
$$
Hence we have \eqref{nashINEQ}.
\end{proof}

\subsection{ Nash's lower bound}

Consider solutions $u$ to the equation \eqref{Leq} which satisfy \eqref{ua}.
Let $v = - \log u$. Then $v$ solves the equation
$$
\partial_t v = \Delta v -  \frac 2 r  \partial_r v  - (b\cdot \nabla)  v - (\nabla v)^2.
$$
We will show that solutions to this equation satisfy one of the fundamental inequalities in the work of Nash.

First we define $\eta(x)$ to be smooth and radial with $\eta=1$ on $B_{1/2}$ and support in $B_1$.  We rescale
$$
 \eta_R (x) = \eta (x/R) R^{-3/2}.
$$
Further suppose that $\int_{\mathbb{R}^3} \eta^2 dx = 1$.  Now we may define the weighted spaces
\[
\| f  \|_{L^p (\eta_R)}^p = \int_{\mathbb{R}^3} |f|^p ~\eta_R^2 dx.
\]
Now we may state the following lemma.

\begin{lemma}\label{NashlowL}
Suppose that for some  $0\le  q < 1$ we have
\begin{equation}\label{bassum2}
\|b(s) \|_{L^2(\eta_R)} \le  CR^{- 1+q} |s|^{-q/2}.
\end{equation}
Then there is a $\delta> 0$ such that
\begin{equation}\label{Nashlow}
- \int_{\mathbb{R}^3}  \log u(x, t) ~\eta_R^2 dx \le C, \quad \text{for} \; - \delta R^2 \le t <0.
\end{equation}
\end{lemma}

Notice that this implies the key step, equation (3.9) in \cite{CSTY},
and thus proves the  H\"older continuity. In fact, since \eqref{Nashlow} holds for
every time, it is stronger than (3.9) which involves time integration.
Further we remark that \eqref{assumption} is enough to ensure \eqref{bassum2} with $q=2\epsilon$ whenever $\epsilon>0$.

\bigskip

\begin{proof}  We first rescale by a factor $R$ for $x$ and $R^2$ for
$t$.  We define $v_R(x,t)=v(Rx, R^2 t)$, which satisfies
$$
\partial_s v_R
=
\Delta v_R -\frac{2}{r}\frac{\partial v_R}{\partial r} - (R b_R \cdot \nabla) v_R  - (\nabla v_R)^2.
$$
Above $b_R(x, t) = b(Rx, R^2 t)$.  Our goal is now to prove that
\begin{equation}\label{Nashlow-1}
- \int_{\mathbb{R}^3}   v_R ~\eta^2 dx \le C,
 \quad \text{for} \; - \delta \le t <0.
\end{equation}
The rescaled version of the assumption on $b$ becomes
\begin{equation}\label{b2-1}
\|Rb_R( \cdot, s) \|_{L^2 (\eta)} \le  C |s|^{-q/2}.
\end{equation}
Since we will only use \eqref{b2-1}, we shall drop all $R$ in the subscript from now on and set $R=1$. We  have
$$
 \int_{\mathbb{R}^3}  \partial_s v ~ \eta^2 dx
 =  \int_{\mathbb{R}^3} \left\{ \Delta v -  \frac 2 r  \partial_r v  - (b\cdot \nabla)  v \right\}~\eta^2 dx
  - \int_{\mathbb{R}^3}   (\nabla v)^2 ~\eta^2 dx.
$$
We will  estimate the terms in parenthesis.

For the first term
we use the Cauchy-Schwartz inequality
$$
\int_{\mathbb{R}^3} \Delta v  ~\eta^2 dx
=
-2 \int_{\mathbb{R}^3} \nabla v \cdot \nabla \eta ~\eta dx
\le \frac{1}{8} \int_{\mathbb{R}^3}   |\nabla v|^2 ~\eta^2 dx
+
8 \int_{\mathbb{R}^3}  |\nabla \eta|^2 ~ dx.
$$
Next let
$$
\bar{v}:= v(s, x) -\langle v\rangle(s),
\qquad
\langle v\rangle(s) :=  \int_{\mathbb{R}^3}  v (s, x) ~\eta^2 dx.
$$
We now consider the middle term inside the parenthesis.
Integrating by parts, we have
\begin{equation*}
\begin{aligned}
- \int_{\mathbb{R}^3}     \frac 2 r   \partial_r v ~\eta^2 dx
=
- \int_{\mathbb{R}^3}     \frac 2 r   \partial_r \bar{v} ~\eta^2 dx
=&
- \int_{-\infty}^{\infty}  2    \bar{v}~ \eta^2 dz \Big |_{r=0}^\infty
+
\int_{\mathbb{R}^3} \frac 2 r \bar{v}  \partial_r \eta^2 dx
\\
\le&
 C-C\langle v\rangle(s)
+
4\int_{\mathbb{R}^3}  \bar{v}  \frac{\partial_r \eta}{r} ~\eta dx.
\end{aligned}
\end{equation*}

We have just used
$$
- \int_{-\infty}^{\infty}  2    \bar{v}~ \eta^2 dz \Big |_{r=0}^\infty
= \int_{-\infty}^{\infty}  2    \bar{v}~ \eta^2 dz \Big |_{r=0}
\le  \int_{-\infty}^{\infty}  2    v(t, z, r=0)~ \eta^2(z, r=0) dz
\le C-C\langle v\rangle(s).
$$
We remark that the constant is $u(t,z,r=0)=a \ge 1$.
Furthermore,
$$
4\int_{\mathbb{R}^3}  \bar{v}  \frac{\partial_r \eta}{r} ~\eta dx
\le
4 \|\bar{v}   \|_{L^2 (\eta)} \| \partial_r \eta / r \|_{L^2(B_1)}
\le C + \frac{1}{8}\|\nabla v   \|_{L^2 (\eta)}^2.
$$
Here we used the spectral gap estimate
$$
 \int_{\mathbb{R}^3} |\nabla v|^2~ \eta^2 dx
 \ge
 c \int_{\mathbb{R}^3} \bar{v}^2  ~ \eta^2 dx.
$$
Finally we consider the last term in parenthesis.
We use the Cauchy-Schwartz inequality together with \eqref{b2-1} to obtain
\begin{gather*}
\int_{\mathbb{R}^3} (b  \cdot\nabla)  v ~ \eta^2 dx
 \le
 \|b  \|_{L^2 (\eta)} \| \nabla v \|_{L^2(\eta)}
 \le 4\|b  \|_{L^2 (\eta)}^2+ \frac{1}{4}\| \nabla v \|_{L^2(\eta)}^2
  \le C  |s|^{-q}+ \frac{1}{4}\| \nabla v \|_{L^2(\eta)}^2.
\end{gather*}
Combining the inequalities in this paragraph
we have
$$
 \int_{\mathbb{R}^3}    \left\{ \Delta v -  \frac 2 r  \partial_r v - (b\cdot \nabla)  v  \right\}~\eta^2 dx
  - \frac{1}{2}\int_{\mathbb{R}^3}   (\nabla v)^2 ~\eta^2 dx
  \le
 C (1+  |s|^{-q}-\langle v\rangle(s))  .
$$
Thus there is a constant $C$ such that
$$
 \int_{\mathbb{R}^3}  \partial_s v (s) ~ \eta^2 dx
 \le
 C (1+  |s|^{-q}-\langle v\rangle(s)) - \frac{1}{2}\int_{\mathbb{R}^3}   |\nabla v|^2 ~\eta^2 dx.
$$
We will use this inequality to prove the lemma.

We plug the Nash inequality \eqref{nashINEQ} with $M=2$ into the inequality above (also using the spectral gap estimate) to obtain
\begin{equation} \label{N2}
 \partial_s \langle v\rangle(s)
 \le
 C (1+  |s|^{-q}-\langle v\rangle(s))
 -  \frac{C \| u(s) \|_{L^1(\eta)}^2}{M^{2}}
 \big   |  \langle v\rangle(s) +  \log \| u(s) \|_{L^1(\eta)}  \big  | ^2.
\end{equation}
This differential inequality will now be manipulated  into a form which we find useful.
For some $\kappa > 0$, \eqref{lb} lets us conclude
\begin{equation} \label{A}
 \| u \|_{L^1(Q_{1/2})} \ge \kappa.
\end{equation}
Let $\chi$ be the characteristic function of the non-empty set
\[
W:= \{ s:  \| u(s) \|_{L^1(B_{1/2})} \ge \kappa/10 \}.
\]
Since $u$ is bounded above by a constant $M$, it follows from \eqref{A} that
\begin{equation}\label{W}
|W| \ge \frac \kappa {10 M}.
\end{equation}
Hence for some $O(1)$ constants $C \ge 1$ and $\gamma >0$ we have
\[
\partial_s \langle v\rangle(s)
 \le
 C (1+ |s|^{-q} -\langle v\rangle(s))-  C \chi (s) \| u(s)\|_{L^1(\eta)} ^2 \big   |  \langle v\rangle(s) +  \log \| u(s) \|_{L^1(\eta)}  \big  | ^2
\]
\begin{equation} \label{N3}
 \le   C (1+ |s|^{-q} -\langle v\rangle(s))-
 \gamma  \chi (s)  \big   |  \langle v\rangle(s) +  \log \| u(s) \|_{L^1(\eta)}   \big  | ^2.
\end{equation}
Notice that, since $q<1$,  this inequality implies for  $s_2 \ge s_1$ that
\[
\langle v\rangle(s_2) \le e^{C|s_1-s_2|}\langle v\rangle(s_1) + Ce^{C|s_2|}.
\]
Therefore we may assume that
\begin{equation}\label{a1}
\langle v\rangle(s) \ge  4 |\log (10/\kappa)|+ 4 C (1+ |s|^{-q}) , \qquad \text{for all} \, -1 \le s \le -  \kappa /(20 M).
\end{equation}
Since, otherwise, we would have $\langle v\rangle(s) \le C_1 $ for some $s_0$ in that range and then for all times later on. This would prove the Lemma.

Under assumption \eqref{a1}, we have for $ -1 \le s \le - \kappa /(20 M)$ and some positive constant $C_1$ that
\begin{equation} \label{N4}
 \partial_s \langle v\rangle(s)
 \le  -  C_1  \chi (s)  \langle v\rangle(s)^2.
\end{equation}
Divide both sides by $\langle v\rangle(s)^2$  and integrate the inequality from $-1$ to $t$.  We have for $ t = - \frac {\kappa } {20 M}$ the following
\begin{equation} \label{N5}
\langle v\rangle(-1)^{-1}- \langle v\rangle(t)^{-1}
 \le  -  \int_{-1}^t  C_1\chi (s) ds  \le - C_2.
\end{equation}
Notice that the range of $t$ and \eqref{W}  guarantee that $C_2 > 0$.
Since by assumption \eqref{a1}, $\langle v\rangle(-1) \ge 0$,
this proves  \eqref{Nashlow} at time $t=- \frac {\kappa } {20 M}$ and hence all the time later  on.
\end{proof}

\subsection{Proof of H\"older continuity}

From Lemma \ref{NashlowL}, there is a $0< \tau <1$ such that  for any $\e > 0$ there is an $\delta$ so that
\begin{equation}\label{keysmallness}
\left|\left\{X \in Q(R,\tau): u(X) \le \delta\right\}\right|
\le
\epsilon |Q(R,\tau)|,
\end{equation}
Let $U=\delta -u$, where $\delta$ is the constant from \eqref{keysmallness}. $U$ is clearly a solution of
\eqref{Leq} and $U|_{r=0} =\delta -a< 0$. So we can
apply Lemma \ref{weakH} to conclude
\begin{equation}\label{eq15}
 \sup_{Q(d/2)} ( \delta - u)  \le \bke{ \frac{C}{|Q(d)|}
\int _{Q(d)} |(\delta -u)^+|^2}^{1/2}.
\end{equation}

Let $d= \sqrt{ \tau } R$
so that
$Q(d) \subset Q(R, \tau)$.  By \eqref{eq15}
and \eqref{keysmallness},
\begin{align*}
\delta - \inf_{Q(d/2)} u & \le \bke{ \frac{C}{|Q(d)|} \int
_{Q(d)} |(\delta -u)^+|^2}^{1/2}
\\
& \le
\left( \frac{C\delta^2 \epsilon |Q(R,  \tau)|}{|Q(d)|}
\right)^{1/2}
=
C \delta \epsilon^{1/2} \left( \tau\right)^{-3/4},
\end{align*}
which is less than $\frac \delta 2$ if $\epsilon$ is chosen sufficiently
small. We conclude
\[
\inf_{Q(d/2)} u \ge \frac \delta 2 .
\]
This is the lower bound we seek.

We define
\[
m_d \equiv \inf_{Q(d/2)} \Gamma,
\quad M_d \equiv \sup_{Q(d/2)} \Gamma.
\]
Then from \eqref{udef} we have
$$
\inf_{Q(d/2)} u =
\left\{
\begin{aligned}
2(m_d - m_2)/M & \quad \text{if}\quad -m_2 > M_2,
\\
2(M_2 - M_d)/M & \quad \text{else},
\end{aligned}
\right.
$$
Notice that both expressions above are non-negative in any case; thus we can
add them together to observe that
$$
\frac \delta 2
\le
\frac{2}{M}\left\{M-\osc(\Gamma,d/2)\right\}.
$$
Here
$\osc(\Gamma,d/2)=M_d-m_d$
and
$\osc(\Gamma,2R)=M_2-m_2=M$.
We rearrange the above
$$
\osc(\Gamma, d/2) \le
\left(1-\frac{\delta}{4} \right)
\osc(\Gamma, 2R).
$$
This is enough to produce the desired H{\"o}lder continuity via the standard argument.

\section{Proof of main theorem}\label{pfOFmainTHM}

In this section we prove Theorem \ref{mainthm} under the assumption
\eqref{assumption1}. It is similar to \cite{CSTY} section 2, which
assumes a stronger assumption $|v|\le C_*(r^2-t)^{-1/2}$.
First we show that our solutions, which satisfy
\eqref{assumption1}, are in fact suitable weak solutions. Recall that
a pair of suitable weak solution $(v,p)$ satisfy
\begin{equation}\label{SWSbound}
v \in L^\infty_t L^2_x(Q), \quad \nb v \in L^2(Q), \quad p \in
L^{3/2}(Q).
\end{equation}
and the local energy inequality.

Fix $\beta \in (1,5/3)$. For $t \in (-T_0,0)$, we have by
\eqref{assumption1}
\[
\int_{\R^3} \frac{|v(x,t)|^4 }{|x|^{\beta}} \,dx \le \int_{\R^3}
\frac {1}{|x|^\beta} \frac{C_* rdr dz}{(r^2-t)^{2-8\e} r^{8\e}
|t|^{4 \e}},
\]
\[
= C \int_{\R^2} \frac{C_* rdr }{(r^2-t)^{2-8\e} r^{\beta-1+8\e}
|t|^{4 \e}} = C C_* |t|^{-(1 + \beta)/2},
\]
where we have used the scaling and $\beta-1+ 8 \e < 2$ so that it
is integrable. Define $R_i$'s to be the Riesz transforms: $R_i =
\frac{\partial_i}{\sqrt{-\Delta}}$. We consider the singular
integral
\[
\td p(x,t) = \int \sum_{i,j}\pd_i  \pd_j (v_i v_j)(y) \frac
1{4\pi|x-y|} \, dy = \sum_{i,j}R_i R_j(v_i v_j ).
\]
Since $|x|^{-\beta}$ is a $A_2$ weight function, we have
\begin{equation}\label{eq4-1}
\int \frac 1{|x|^{\beta}} |\td p(x,t)|^2 \,dx \le c\int \frac
1{|x|^{\beta}} |v(x,t)|^4 \,dx \le c+c|t|^{-(1+\beta)/2} .
\end{equation}
With this estimate, the same argument as in \cite{CSTY}
proves that $p(x,t)=\td p(x,t)$ for all $x$ and for almost every
$t$. Moreover, from \eqref{eq4-1} and $\beta<5/3$ we conclude that
\begin{equation}
\int_{Q_{1}} |p(x,t)|^{3/2}\,dxdt \le c \int_{-1}^0   \bke{
\int_{B_{1}}  \frac 1{|x|^{\beta}} |p(x,t)|^2 \,dx}^{3/4} dt \le C
. \label{p32}
\end{equation}

Since $\e$ is arbitrarily small, the pointwise estimate
\eqref{assumption1} on $v$ implies
\begin{equation}\label{vbd}
v \in L^s_tL^q_x(Q_1), \quad \frac{1}{q}+\frac{1}{s}> \frac{1}{2},
\quad q<\frac{1}{\e}.
\end{equation}
We will use $(s,q)=(3,3)$. Thus the vector product of \eqref{nse}
with $u \ph$ for any $\ph \in C^\infty_c(Q_{1})$ is integrable in
$Q_{1}$ and we can integrate by parts to get
\begin{equation}
 2\int_{Q} |\nabla v|^2 \varphi
= \int_{Q} \left\{|v|^2(\partial_t\varphi+\Delta \varphi) +
 ( |v|^2+2p)v\cdot\nabla\varphi\right\}, \quad \forall \ph \in
C^\infty_c(Q), \ \ph \ge 0. \label{localE}
\end{equation}

For any $R \in (0,1)$ and $t_0 \in (-R^2,0)$, we can further
choose a sequence of $\ph$ which converges a.e. in $Q_R$ to the
Heviside function $H(t_0-t)$. Since the limit of $\pd_t\ph$ is the
negative delta function in $t$, this gives us the estimate
\begin{equation}\label{SWS-v}
\mathop{\mathrm{ess\, sup}}_{-R^2<t<0}\int_{B_R}|v(x,t)|^2dx +
\int_{Q_R} |\nb v|^2 \le C _{R} \int _{Q_1} (|v|^3 + |p|^{3/2}).
\end{equation}
These estimates show that $(v,p)$ is a suitable weak solution of
\eqref{nse} in $Q_{R}$. Note that these bounds depend on $C_*$ of
\eqref{assumption1} only, not on $\norm{p}_{L^{5/3}(\R^3 \times
(-T_0,0))}$.

To prove Theorem \ref{mainthm}, it suffices to show that every
point on the $z$-axis is regular.  Suppose now a point
$x_*=(0,0,x_3)$ on the $z$-axis is a singular point of $v$.
Without loss of generality, we assume that $x_3=0$. We will use
the following regularity criterion, a variant of the criterion in
\cite{MR673830} and proven in \cite{MR1685610}, to obtain a
contradiction.

\begin{lemma}\label{th6-1}
Suppose that $(v,p)$ is a suitable weak solution of \eqref{nse} in
$Q(X_0,1)$.  Then there exists an $\e_1>0$ so that $X_0$ is a
regular point if
\begin{equation} \label{eq2-4-1}
\limsup_{R \downarrow 0} \frac{1}{R^2} \int_{Q(X_0,R)} |v|^3 \le
\e_1.
\end{equation}
\end{lemma}

Let $(v^\la,p^\la)$ be the rescaled solutions of \eqref{nse}
defined by
\begin{equation}
\label{rescale-1} v^\la(x,t) = \la v(\la x, \la^2 t),\quad
p^\la(x,t) = \la^2 p (\la x, \la^2 t).
\end{equation}
For $(v^\la,p^\la)$ with $0<\la<1$, the pointwise estimate
\eqref{assumption1} is preserved:
\[
|v^\la(x,t)| \le C_* (r^2-t)^{-1/2+2\e}|t|^{-\e } r^{-2\e}, \quad
(x,t) \in \R^3 \times (-T_0,0).
\]
Fix $R_*>0$. Since we assume $x_*$ is a singular point, by Lemma
\ref{th6-1} there is a sequence $\la_k$, $k\in \N$, so that $\la_k
\to 0$ as $k \to \infty$ and
\begin{equation}
\label{blowup0} \frac{1}{R_*^2} \int_{Q_{R_*}} |v^{\la_k}|^3 =
\frac{1}{(R_*\la_k)^2} \int_{Q(x_*,R_*\la_k)} |v|^3
>\e_1.
\end{equation}
We will derive a contradiction to this statement.

Since $(v^\la,p^\la)$ satisfies the pointwise estimate
\eqref{assumption1}, we have $v^\la \in L^q(Q_1)$ for $q\in
(1,4)$. Moreover, the same argument as above provides the uniform
bounds for $R<1$:
\begin{equation}\label{u-bd}
\quad \mathop{\mathrm{ess\, sup}}_{-R^2<t<0}
\int_{B_R}|v^\la(x,t)|^2dx +  \int_{Q_R} |\nb v^\la|^2 \le
c\int_{Q_1}   |v^\la |^3 + |p^\la|^{3/2}  \le C.
\end{equation}
Following the same argument in section 2.3 and section 2.4 of
  \cite{CSTY}, we conclude from these estimates that
there is a subsequence of $(v^{\la_k},p^{\la_k})$, still denoted
by $(v^{\la_k},p^{\la_k})$, weakly converges to some suitable weak
solution $(\lv,\lp)$ of the Navier-Stokes equations in $Q_R$ and
\[
v^{\la_k} \to \lv,
\]
strongly in $L^q(Q_R)$ for all $1\le q < 4$ and $0<R<1$.

Now the H{\"o}lder continuity of $\Gamma=rv_\theta$ at $t=0$
proven in Section 5 implies
\[
\int_{Q_R} \left|v^\la_\theta\right| \le C\la^\alpha \to 0 \quad
\text{as} \quad \la \downarrow 0
\]
for some $\alpha>0$. Thus the limit $\lv$ has no-swirl,
$\lv_\theta =0$.

Let $\bar \omega = \nb \times \lv$ be the vorticity of $\lv$ and
$\bar \omega_\theta =\pd_z \lv_r- \pd_r \lv_z$ be the $\theta$
component of $\bar \omega$. The function $\Om = \bar
\omega_\theta/r$ solves
\begin{equation}\label{Om2}
\left(\pd_t + \bar{b} \cdot \nb - \Delta - {\frac 2r} \pd_r
\right) \Om=0,
\end{equation}
where $\lv_\th=0$ is used and $\bar{b} = \lv_r e_r + \lv_z e_z=
\lv. $

Since $\bar{v}$ is the limit of $v^{\la_k}$, it satisfies
\eqref{assumption1} and also satisfies \eqref{vbd}. Following the
argument of section 2.4 in \cite{CSTY}, we conclude
from \eqref{vbd} and the estimates for the Stokes system that
\[ \norm{ \nb \lv}_{L^{5/4}_tL^{5/2}_x
(Q_{5/8})} \le C \norm{  \lv}_{ L^{5/2}_tL^{5}_x (Q_{3/4})}^2 + C
\norm{  \lv}_{ L^{5/4}(Q_{3/4})} \le C.
\]
Hence $\Om$ has the bound
\begin{equation}\label{Om3}
\norm{\Om}_{L^{20/19}(Q_{5/8})} \le \norm{ \nb
\lv}_{L^{5/4}_tL^{5/2}_x (Q_{5/8})}
\norm{1/r}_{L^{\infty}_tL^{20/11}_x (Q_{5/8})} \le C.
\end{equation}

Since $\bar{b}$ satisfies \eqref{assumption1}, it also satisfies
\eqref{assumption}. We conclude from the local maximum estimate
Lemma \ref{weakH} and \eqref{Om3} that
\[
 \Om \in L^\infty(Q_{5/16}).
\]
Furthermore we know that $\curl \lv = \bar \omega_\th e_\th \in
L^\infty(Q_{5/16})$ from the above estimate on $\Om$ since
$\bar{v}_\theta=0$. Now we can apply the div-curl estimate
\begin{equation*}
\norm{\nb^{k+1} v}_{L^q(B_{R_2})} \le c \norm{\nb^k \div
v}_{L^q(B_{R_1})} + c \norm{\nb^k \curl v}_{L^q(B_{R_1})} + c
\norm{ v}_{L^1(B_{R_1})},
\end{equation*}
to obtain $L^\infty$ estimate for $\lv$.  Since $\div \lv = 0$ and
$\lv \in L^\infty_t L^1_x(Q_{5/16})$ by \eqref{assumption1}, we
thus conclude $\nb \lv \in L^\infty_t L^4_x(Q_{1/4})$ by taking
$q=4$ and $k=0$ in the div-curl estimate. By the Sobolev
embedding, we have $\lv \in L^\infty(Q_{1/4})$.

Now we can deduce regularity of the original solution from the
regularity of the limit solution. We have shown that
$$
\left| \lv(x,t) \right| \le C^\prime_* \quad \text{in} \quad
Q_{1/4}.
$$
Above $C_*^\prime$ depends upon $C_*$ but not on the subsequence
$\lambda_k$.  Since the constant can be tracked, we may initially
choose $R_*$ sufficiently small to guarantee that
\[
\frac{1}{R_*^2} \int_{Q_{R_*}} |\lv|^3 \le \e_1/2,
\]
where $\e_1$ is the small constant in Lemma \ref{th6-1}. Since
$v^{\la_k} \to \lv$ strongly in $L^3$, for $k$ sufficiently large
we have
$$
\frac{1}{R_*^2} \int_{Q_{R_*}}  |v^\lambda|^3 \le \frac{1}{R_*^2}
\int_{Q_{R_*}}  |\bar v|^3 + \frac{1}{R_*^2} \int_{Q_{R_*}}
|v^\lambda- \bar v |^3 \le \e_1.
$$
But this is a contradiction to \eqref{blowup0}. We have proved
Theorem \ref{mainthm}.

%%%%%%%%%%%%%%%%%%%%%%%%%%%%%%%%%%%%%%%%%%%%%%%%%%%%%%%%%%%%%%%%%%
\section*{Appendix: The case $\e=0$}
\addcontentsline{toc}{section}{Appendix}
\setcounter{theorem}{0}
\renewcommand{\thetheorem}{A.\arabic{theorem}}
\setcounter{equation}{0}
\renewcommand{\theequation}{A.\arabic{equation}}

In this appendix we prove Theorem \ref{mainthm} under the assumption
$|v(x,t)|\le C_* r^{-1}$, the $\e=0$ case. The argument in this
appendix was obtained after a preprint of \cite{KNSS} appeared.  Note
that this argument does not take scaling limits and all bounds are
computable.

Let $M$ be the maximum of $|v|$ up to a fixed time $t_{1}$. We will
derive an upper bound of $M$ in terms of $C_*$ and independent of
$t_1$.  We may assume $M>1$. Define
%$$ v_\theta^M (X, T) = M^{-1} v_\theta(X/M, T/M^{2}), \quad X=(R, Z).
%$$
\[
v^M (X, T) = M^{-1} v(X/M, T/M^{2}), \quad X=(X_1,X_2, Z).
\]
For $x=(x_1,x_2,z)$ and $X=(X_1,X_2, Z)$, let
$r=(x_1^2+x_2^2)^{1/2}$ and $R=(X_1^2+X_2^2)^{1/2}$. We have the
following estimates for all $r$ and $R$ for time $t \le t_{1}$ and
$T \le M^2t_{1}$: \be \label{eq:A1} |v (x, t)| \le C/ r, \quad
|v^M (X, T)| \le
 C/ R, \quad
 |\nabla^k  v^M| \le C_k.
\ee The last inequality follows from $\norm{v^M}_{L^\infty}\le 1$
for $t<t_1$ and the regularity theorem of Navier-Stokes equations.
Its angular component (we omit the time dependence below)
$v^M_\theta(R,Z)$ satisfies $ v_\theta^M (0, Z)=0=\partial_Z
v_\theta^M (0, Z)$ for all $Z$. By mean value theorem and
\eqref{eq:A1},
 $|v_\theta^M(R, Z)| \le C
R$ and $|\partial_Z v_\theta^M(R, Z)| \le C R$ for
$R\le 1$.
Together with \eqref{eq:A1} for $R\ge 1$, we get
\[
|v_\th^M| \le C\min(R,R^{-1}),\quad
|\pd_Z v_\th^M| \le C\min(R,1), \quad \text{for} \; R>0.
\]
By \cite[Theorem 3.1]{CSTY}, under the assumption $|v(x,t)|\le C_*
r^{-1}$, $\Gamma= rv_\theta$ satisfies $|\Gamma(r,z,t)| \le C
r^{\alpha}$ when $r$ is small, uniformly in $r$ and $t$, for some $C$ and
small $\alpha >0$ depending on $C_*$. Thus,
$ |v_\theta^M (R, Z)| \le C R^{-1+ \alpha} M^{-\alpha}$ for $R>0$.
From these estimates we have
\be\label{x2} \frac {
|\partial_Z (v_\theta^M)^2 (R)| } {R^2} \le
\frac{C\min(R,R^{-1+ \alpha} M^{-\alpha})\cdot \min(R,1)}{R^2} \le
\frac { C}  {R^{3-
\alpha}M^\al + 1}. \ee

Consider now the angular component of the rescaled vorticity.
Recall $\Om = \om_\th/r$. Let
$$
f=\Omega^M (X, T) = \Omega (X/M,  T/M^2) M^{-3} = \omega^M_\th (X, T)/R.
$$
Since $\omega^{M}_\th$ and $\nb \om^M_\th$ are bounded by
\eqref{eq:A1} and $\om^M_\th|_{R=0}=0$, we have $|f| \le C(1+R)^{-1}$.
From the equation of $\omega_\th$ (see \eqref{q-eq}), we have
\[
(\pd_T - L)f = g, \quad L = \Delta+ \frac{2}{R} \partial_{R} - b^M
\cdot \nabla_X ,
\]
where $g=R^{-2} \partial_Z (v_\theta^M)^2 $ and $b^M = v_r^M
e_R + v_z^M e_Z$.
Let $P(T, X; S, Y) $ be the evolution kernel for $ \partial_{T}- L$.
By Duhamel's formula
$$
f(X,T) = \int  P(T, X; S, Y)  f  ( Y,S) d Y+ \int_{S}^{T}  \int
P(T, X; \tau, Y) g (Y,\tau) d Y d\tau =: I + II.
$$
By Carlen and Loss \cite{MR1381974}, in particular its equation
(2.5), the kernel $P$ satisfies $P\ge 0$, $\int P(T, X; S, Y) d Y
= 1$ and, using $\|b^M \|_{\infty} < 1$,
\[
P(T, X; S, Y) \le C (T-S)^{-3/2} e^{- h(|X-Y|,T-S)}, \quad h(a,T) =
  C\frac{a^2}T (1-\frac Ta)_+^2.
\]
Using $e^{-h(a,T)} \le C e ^{-Ca/T}$ we get for $X=(X_1,X_2,X_3)$
and $Y=(Y_1,Y_2,Y_3)$
\be %\tag{II}
P(T, X; S, Y) \le C (T-S)^{-3/2} e^{- C |X_3-Y_3|/(T-S)}.
\ee
Here we only assert the spatial decay in the $X_3$ direction so that the proof
of \cite{MR1381974}, where the term $R^{-1}\pd_R$ in $L$ is not present,
needs no revision.
With these bounds %, using the coordinate  $Y=(R, Z)$
and  H\"older inequality we get
\begin{equation}
\label{X4}
|I| \le \left [ \int P(T, X; S, Y) |f(Y,S )|^{3} d Y \right ]^{\frac 13}
\le
\left [ C (T-S)^{-\frac 32} \int %e^{- C |X_3-Y_3|/(T-S)}
e^{-C \frac {|X_3-Y_3|}{T-S}}
\frac {RdR}{(1+R)^3}
 d Z \right ]^{\frac 13}
    \le  C (T-S)^{-1/6},
\end{equation}
and
\begin{equation}
\label{X5}
|II| \le \int^{T}_{S} \int (T-\tau)^{-3/2} e^{- C\frac{|
X_3-Y_3|}{T-\tau}}\frac {R \, d R} {R^{3- \alpha} M^{\alpha}+ 1} \,
d Z \,d \tau \le C (T-S)^{1/2} M^{-2\alpha/3 }.
\end{equation}
Combining these two estimates and choosing  $S=T- M^{  \alpha}>-T_0M^2$
(hence $f$ is defined), we have $| f ( X,T)|   \le C M^{-\alpha/6}$.
Thus
$$
|\om_\th (x, t)| \le  |\om_\th^{M} (rM, z M, tM^2)| M^{2} \le
|\Omega^{M}(rM, z M, tM^2)| M^{2} r M \le CM^{3-\alpha/6} r.
$$
Therefore,  we have
\be\label{ob1} |\om_\th (x, t)| \le
CM^{2-\alpha/12}, \quad \text{for}\; r \le M^{-1+\alpha/12}. \ee

Let $b = v_re_r + v_ze_z$ and $B_{\rho}(x_0)=\{x:|x-x_0|<\rho\}$.  By
\eqref{vector-id} and \eqref{b-eq}, $b$ satisfies $ -\Delta b=\curl
(\om_\th e_\theta)$, and hence the following estimate with $p>1$ (see
e.g., \cite{MR0737190}, Theorem 8.17)
\begin{equation}
\label{b01}
\sup_{B_{\rho}(x_0)}|b|\le C(\rho^{-\frac
3{p}}||b||_{L^p(B_{2\rho}(x_0))}+\rho\sup_{B_{2\rho}(x_0)}|\om_\th|).
\end{equation}
Let $\rho=M^{-1+\alpha/24}$, $x_0\in \{(r,\theta,z):r<\rho\}$ and
$1<p<2$. By the assumption $|v| \le C/r$,
\[
\rho^{-\frac 3{p}}||b||_{L^p(B_{2\rho}(x_0))}\le
C\rho^{-\frac 3{p}}||1/r||_{L^p(B_{2\rho}(x_0))}\le
\frac C{2-p}\rho^{-1} \le C(p)M^{1-\alpha/24}.
\]
This together with \eqref{ob1},
\eqref{b01} and the fact
$|v_\th| = M |v_\th^M| \le
M C\min(R, R^{-1+\al}M^{-\al})$
imply
\begin{equation}\label{M1x}
|v(x, t)| \le C M^{1-\alpha/24}, \quad \text{for } r \le M^{-1+
\alpha/24}.
\end{equation}
On the other hand, the assumption $|v|\le C/r$ implies $|v|\le C
M^{1-\alpha/24}$ for $r\ge M^{-1+ \alpha/24}$. Since $M$ is the
maximum of $v$, this gives an upper bound for $M$.  \qed

%\bigskip

\section*{Acknowledgments}

The authors would like to thank the National Center for Theoretical
Sciences at Taipei  and National Taiwan University for hosting part
of our collaboration. The research of Chen is partly supported by
the NSC grant 95-2115-M-002-008 (Taiwan). The research of Strain is
partly supported by the NSF fellowship DMS-0602513 (USA).  The
research of Tsai is partly supported by an NSERC grant (Canada). The
research of Yau is partly supported by the NSF grants DMS-0602038
and DMS-0804279 (USA).

%===============================================================

%\centerline{{\bf \today}}

%\bigskip

\end{document}